\documentclass[11pt,a4paper]{article}

\usepackage{graphicx}
\usepackage{amsmath}
\usepackage{graphicx}
\usepackage{verbatim}
\usepackage{color}
\usepackage{subfigure}
\usepackage{float}
\usepackage{authblk}

\usepackage{amsfonts}
\usepackage{textcomp}
\usepackage{eurosym}

\usepackage{dynkin-diagrams}
\tikzset{big arrow/.style={
		-Stealth,line cap=round,line width=1mm,
		shorten <=1mm,shorten >=1mm}}

\usepackage{textcomp}
\usepackage{eurosym}
\usepackage{array}
\usepackage{amssymb}
\usepackage{subfigure}
\usepackage{enumitem}

\usepackage{amsthm, amsmath}
\theoremstyle{plain}
\theoremstyle{definition}
\newtheorem{thm}{Theorem}
\newtheorem{cor}{Corollary}

\newtheorem{prop}{Proposition}

\newtheorem{lem}{Lemma}
\newtheorem{remark}{Remark}

\begin{document}
\title{A presentation for a submonoid of the  symmetric inverse monoid}
\author[1]{Apatsara Sareeto}
\author[2*]{Jörg Koppitz}

\affil[1]{Institue of Mathematics, University of Potsdam, Potsdam, 14476, Germany  (E-mail: channypooii@gmail.com)}
\affil[2*]{Corresponding author. Institute of Mathematics and Informatics, Bulgarian Academy of Sciences, Sofia, 1113, Bulgaria  (E-mail: koppitz@math.bas.bg).}

\maketitle
\begin{abstract} A fully invarient congruence relations on the free algebra on a given type induces a variety of the given type. In contrast, a congruence relation of the free algebra provides algebra of that type. This algebra is given by a so-called presentation. In the present paper, we deal with an important class of algebras of type $(2)$, namely with semigroups of transformations on a finite set. Here, we are particularly interested in a presentation of a submonoid of the symmetric inverse monoid $I_n$.  Our main result is a presentations for $IOF_n^{par}$, the monoid of all order-preserving, fence-preserving, and parity-preserving transformations on an $n$-element set.



\end{abstract}

\section{Introduction}
Let $n$ be a positive integer and let $\overline{n}$ be a finite set with $n$ elements, say $\overline{n}=\{1,...,n\}$.   We denote by $PT_n$ the monoid (under composition) of all partial transformations on $\overline{n}$. A partial injection $\alpha$ on the set $\overline{n}$ is a one-to-one function from a subset $A$ of $\overline{n}$, into $\overline{n}$. The domain of $\alpha$ is the set $A$, denoted by $dom(\alpha)$. The range of $\alpha$ is denoted by $im(\alpha)$. The empty transformation will be denoted by $\varepsilon$, it is the transformation with  $dom(\varepsilon)=\emptyset$. The symmetric inverse monoid, denoted by $I_n$, is
the inverse monoid of all injective partial transformations on $\overline{n}$. The
symmetric group, denoted by $S_n$, is the group of all permutations on $\overline{n}$, and is the
group of units of $I_n$. \\
\indent Moore (1897) found a monoid presentation of $S_n$. Since
then, there has been a widespread interest in the subject of finding presentations
of groups and semigroups related to $S_n$. See for example A\v{i}zen\v{s}tat (1958, 1962),
Fernandes (2001), Popova (1961), and Solomon (1996) and references therein.  Also, presentations for other important semigroups were found. Let's give some examples:
A presentation for the Brauer monoid was found by Kudryavtseva and
Mazorchuk \cite{gana}, and for its singular part by Maltcev and Mazorchuk \cite{mal}. For the
partition monoid and its singular part, presentations were found by East \cite{eae, jamm}.
FitzGerald \cite{zit} provided a presentation for the factorizable part of the dual symmetric inverse monoid. \\
\indent Semigroups of order-preserving transformations have long been considered in the literature. A short, and by no means comprehensive, history follows.  A\v{i}zen\v{s}tat \cite{a} and Popova \cite{popo} exhibited a presentation for $O_n$, the monoid of all order-preserving full transformations on an $n$-chain, and for $PO_n$, the monoid of all order-preserving partial transformations on an $n$-chain.  Solomon \cite{sol} established a presentation for $PO_n$. In  1996, T. Lavers
gave a presentation on the monoid of ordered partitions of a natural number,
which is (up to an isomorphism) a monoid whose elements are order-preserving transformations, and  Catarino \cite{cat2} found a presentation for the
monoid $OP_n$ of all orientation-preserving full transformations of an $n$-chain
(see also \cite{cat}). The injective counterpart of $OP_n$, i.e. the monoid $POPI_n$ of all injective orientation-preserving partial transformations on a chain with $n$ elements, was studied by Fernandes (2000).
A	Presentation for the monoid $POI_n$ and its extension $PODI_n$, the monoid of all injective order-preserving or order-reversing partial transformations on an $n$-chain, was given by Fernandes in 2001 and Fernandes et al. in 2004, respectively. See also \cite{fer}, for a survey on known presentations for various transformation monoids.   \\
\indent East \cite{east}  found a presentation of the singular part of a symmetric inverse monoid. Later, Easdown et al. studied a presentation for the dual symmetric inverse monoid in \cite{da}. East \cite{ea2} showed a presentation for the singular part of the full transformation semigroup. In \cite{ganna}, Kudryavtseva et al. studied a presentation for the partial dual symmetric inverse monoid, and in the same year, East \cite{jam} studied a symmetrical presentation for the singular part of the symmetric inverse monoid. Fernandes and Quinteiro showed presentations for monoids of finite partial isometries in \cite{feh}. Currently, Koppitz and Worawiset \cite{jl} showed ranks and presentations for order-preserving transformations with one fixed point. \\
\indent Now, we  consider the linear order $1<2<\cdot\cdot\cdot<n$ on $\overline{n}$. We say that a transformation $\alpha\in PT_n$ is order-preserving if $x< y$ implies $x\alpha\leq y\alpha$, for all $x,y\in dom(\alpha)$. \\
\indent A non-linear order, closed to linear order in some sense, is the so-called zig-zag order. The pair $(\overline{n}, \preceq)$ is called zig-zag poset or fence if 
\begin{center}
	$1 \prec 2 \succ \cdot\cdot\cdot \prec n - 1 \succ n$ or $1 \succ 2 \prec \cdot\cdot\cdot \succ n - 1 \prec n$ if n is odd \\
	and  $1 \prec 2 \succ \cdot\cdot\cdot \succ n - 1 \prec n$ or $1 \succ 2 \prec \cdot\cdot\cdot \prec n - 1 \succ n$ if n is even.
\end{center}
The definition of the partial order $\preceq$ is self-explanatory. Transformations on fences were first considered by Currie and Visentin in 1991  as well as Rutkowski in 1992. We observe that every element in a fence is either minimal or maximal. Without loss of generality, let $1 \prec 2 \succ 3 \prec\cdot\cdot\cdot \succ n$ and $1 \prec 2 \succ 3 \prec\cdot\cdot\cdot \succ n-1\prec n$, respectively. Such fences are also called up-fences. The fence $1 \succ 2 \prec 3 \succ \cdot\cdot\cdot \prec n$ and $1 \succ 2 \prec 3 \succ \cdot\cdot\cdot \prec n-1\succ n$, respectively, would be called down-fence. We observe that any $x, y \in \overline{n}$ are comparable if and only if $x \in \{y - 1, y, y + 1\}$. \\ \indent We say a transformation $\alpha \in I_n$ is fence-preserving if $x \prec y$ implies  $x\alpha \prec y\alpha$, for all $x, y \in dom (\alpha)$. We denote by $PFI_n$ the submonoid of $I_n$ of all  fence-preserving partial injections of $\overline{n}$. Fernandes et al. characterized the full transformations on $\overline{n}$ preserving the zig-zag order \cite{fer 2}. It is worth mentioning that several other properties of monoids of fence-preserving full transformations were also studied.  We denote by $IF_n$ the inverse subsemigroup of all regular elements in $PFI_n$.  Fence-preserving transformations are also studied in \cite{fer 2, Jen, Loh, Sir}. For general background on semigroups and standard notations, we refer the reader to \cite{clif, Howie}. \\ \indent Our focus in this paper is the study of a submonoid of $POI_n \bigcap IF_n$, namely the monoid $IOF_n^{par}$ of all $\alpha\in POI_n \bigcap IF_n$ such that $x$ and $x\alpha$ have the same parity for all $x\in dom (\alpha)$. It is easy to verify that $IOF_n^{par}$ forms a monoid. In \cite{Apa}, the elements of $IOF_n^{par}$ are characterized: 
\begin{prop} \label{4 choice}
	Let  $p\leq n$ and let $\alpha = \bigl(\begin{smallmatrix}
	d_1 &<&d_2&<&   \cdots & < &d_p \\
	m_1 & &m_2 & & \cdots &    & m_p
	\end{smallmatrix}\bigr) \in I_n$. Then $\alpha\in IOF_n^{par}$ if and only if  the following four conditions hold:\\
	(i) $m_1<m_2<...<m_p$. \\
	(ii) $d_1$ and $m_1$ have the same parity. \\
	(iii) $d_{i+1}-d_i=1$ if and only if $m_{i+1}-m_i=1$ for all $i\in\{1,...,p-1\}$.  \\
	(iv) $d_{i+1}-d_i$ is even if and only if $m_{i+1}-m_i$ is even for all $i\in\{1,...,p-1\}$.
\end{prop}
\indent Moreover, the authors of that paper have already determined the rank of $IOF_n^{par}$ and have provided a minimal generating set (see \cite{Apa}). Our target is to exhibit a monoid presentation for $IOF_n^{par}$. \\
\indent Let ${X}$ be a set and denoted by $X^*$ the free monoid generated by $X$. A monoid presentation is an ordered pair $\langle X \ \vert \ R \rangle$, where $X$ is an alphabet and $R$ is a subset of $X^*\times X^*$. An element $(u,v)$ of $X^*\times X^*$ is called a relation and it is represented by $u\approx v$. The monoid $IOF_n^{par}$ is said to be defined by a monoid presentation (or has a monoid presentation) $\langle X \ \vert \ R \rangle$ if $IOF_n^{par}$ is isomorphic to $X^*/\rho_R$, where $\rho_R$ denotes the smallest congruence on $X^*$ containing $R$. We say that $u\approx v$, for $u,v\in X^*$, is a consequence of $R$ if $(u,v)\in \rho_R$. For more detail see \cite{ll} or \cite{r}. Since $IOF_n^{par}$ is a finite monoid, we can always exhibited a presentation for it. A usual method to find some presentations is the Guess and Prove Method described by the following theorem adapted to monoids from Ru\v{s}kuc (1995, Proposition 3.2.2).   

\begin{thm} \label{kc}
	Let $X$ be a generating set for $IOF_n^{par}$. Let $R\subseteq X^* \times X^*$ a set of relations and $W \subseteq X^*$ that the following conditions are satisﬁed: \\
	1. The generating set $X$ of $IOF_n^{par}$ satisﬁes all the relations from $R$; \\
	2. For each word $w\in X^*$, there exists a word $w'\in W$ such that the relation $w\approx w'$ is a consequence of $R$; \\
	3. $\lvert W \rvert \leq \lvert IOF_n^{par}\rvert  $. \\
	Then $IOF_n^{par}$ is deﬁned by the presentation $\langle X \ \vert \ R \rangle$.
\end{thm}

\section{Preliminaries}

In this section, we gather the preliminary material we will need in the present paper.  Let $\overline{v}_i$ be the partial identity with the domain $\overline{n}\backslash \{i\}$ for all $i\in\{1,...,n\}$. Further, let 	

\begin{center}   
	$\overline{u}_i = \begin{pmatrix}
	1 &  \cdots  & i & i+1 & i+2 & i+3 & i+4 &  \cdots   & n \\
	3 &   \cdots   & i+2 & - & -& - & i+4 &  \cdots     & n 
	\end{pmatrix}$ \end{center}  and  $\overline{x}_i=(\overline{u}_i)^{-1}$ for all $i\in\{1,...,n-2\}$. By Proposition \ref{4 choice}, it is easy to verify that $\overline{u}_i$ as well as $\overline{x}_i$, $i\in\{1,...,n-2\}$, belong to $IOF_n^{par}$. In \cite{Apa},  the authors have shown that $\{\overline{v}_1,\overline{v}_2,...,\overline{v}_n,\overline{u}_1,\overline{u}_2,...,\overline{u}_{n-2},\overline{x}_1,\overline{x}_2,...,\overline{x}_{n-2}\}$ is a generating set of $IOF_n^{par}$. In order to use Theorem \ref{4 choice}, we define an alphabet,
\begin{center}
	$X_n=\{v_1,v_2,...,v_n,u_1,u_2,...,u_{n-2},x_1,x_2,...,x_{n-2}\}$,
\end{center}  which corresponds to a set of generators of $IOF_n^{par}$. For $w=w_1...w_m$ with $w_1,...,w_m\in X_n$ and $m$ is a positive integer, we write $w^{-1}$ for the word $w^{-1}=w_m...w_1$. We fix a particular sequence of letters as follows:
$x_{i,j}=x_ix_{i+2}...x_{i+2j-2}$ and $u_{i,j}=u_iu_{i+2}...u_{i+2j-2}$ for $i\in\{1,...,n-2\}, j\in\{1,...,\lfloor \frac{n-i}{2} \rfloor\}$ and obtain the following sets of words: 
$W_x=\{x_{i,j} : i\in\{1,...,n-2\},j\in \{1,...,\lfloor \frac{n-i}{2} \rfloor\}\}$, $W_x^{-1}=\{x_{i,j}^{-1}: x_{i,j}\in W_x\}$, and
$W_u=\{u_{i,j} : i\in\{1,...,n-2\},j\in \{1,...,\lfloor \frac{n-i}{2} \rfloor\}\}$.  Let $w$ be any word of the form $w=w_1...w_m$ with   $ w_1,...,w_m\in W_x\cup W_u$ and $m$ is a positive integer. For $k\in\{1,...,m\}$, the word $w_k$ is of the form \begin{align*}
w_k= \begin{cases}
u_{i_k,j_k} & \mbox{if} \ w_k\in W_u \\
x_{i_k,j_k} & \mbox{if} \ w_k\in W_x.
\end{cases}
\end{align*} 
We observe $j_k=\lvert w_k \rvert$, i.e. $j_k$ is the length of the word $w_k$. We define two sequences $1_x, 2_x,..., m_x$ and $1_u,2_u,...,m_u$ of indicators: for $k\in\{1,...,m\}$ let
\begin{align*}
k_x 
=& \begin{cases}
i_k+2\lvert w_k\rvert+2\lvert W_u^k\rvert-2\lvert W_x^k\rvert & \mbox{if}\ w_k\in W_u \\
i_k &  \mbox{if} \ w_k\in W_x
\end{cases}
\end{align*}  and \begin{align*}
k_u 
=& \begin{cases}
i_k+2\lvert w_k\rvert-2\lvert W_u^k\rvert+2\lvert W_x^k\rvert & \mbox{if}\ w_k\in W_x \\
i_k &  \mbox{if} \ w_k\in W_u,
\end{cases}
\end{align*}  where $W_u^k\ (W_x^k)$ means the word $w_{k+1}...w_m$ without the variables in $\{x_1,...,x_{n-2}\}$ \\ (in $\{u_1,...,u_{n-2}\} )$.  Let $Q_0$ be the set of all words $w=w_1...w_m$ with $w_1,...,w_m\in W_x\cup W_u$ and $m$ is a positive integer such that:

\begin{enumerate} [label = {$({\arabic*}_q)$}]
	\item \label{1W} If $w_k, w_l\in W_x$ then $i_k+2j_k+1<i_l$ for $k<l\leq m$; 
	\item \label{2W} If $w_k, w_l\in W_u$ then $i_k+2j_k+1<i_l$ for $k<l\leq m$;
	\item \label{3W} If $w_k\in W_u$ then $i_k+2j_k+2\leq(k+1)_u$ for $k\in\{1,...,m-1\}$  and  $(k+1)_x-k_x\geq 2$; 
	\item \label{4W}  If $w_k\in W_x$ then
	$i_k+2j_k+2\leq(k+1)_x$ for $k\in\{1,...,m-1\}$ and $(k+1)_u-k_u\geq 2$. \\
\end{enumerate}
Let now $w=w_1...w_m\in Q_0$. Then let $w^*=W_u^0(W_x^0)^{-1}$. Further, we define recursively a set $A_w$:
\begin{enumerate} [label = {$({\arabic*}_q)$}, start=5]
	\item \label{5W} 
	If $m_u>m_x$ and $m_u+2\leq n$ then $A_m=\{m_u+2,...,n\}$, \\
	if $m_u<m_x$ and $m_x+2\leq n$ then $A_m=\{m_x+2,...,n\}$, \\   
	otherwise $A_m=\emptyset$;
	\item \label{6w} If $w_k\in W_u$ then $A_k=A_{k+1}\cup\{i_k+2j_k+2,...,(k+1)_u-1\}$ for $k\in\{1,...,m-1\}$, \\
	if  $w_k\in W_x$ then  $A_k=A_{k+1}\cup\{k_u+2,...,(k+1)_u-1)\}$ for $k\in\{1,...,m-1\}$; 	
	\item \label{7W} If $1\in\{1_x,1_u\}$ then $A_w=A_1$, \\
	if $1<1_u\leq 1_x$ then $A_w=A_1\cup\{1,...,1_u-1\}$, \\
	if $1<1_x<1_u$ then $A_w=A_1\cup\{1_u-1_x+1,...,1_u-1\}$.
\end{enumerate}
 For a set $A=\{i_1<i_2<\cdot\cdot\cdot <i_k\}\subseteq \overline{n}$, let $v_A=v_{i_1}v_{i_2}...v_{i_k}$ for some $k\in\{1,...,n\}$. Note that $v_\emptyset$ means the empty word $\epsilon$. For convenience, we put  $v_i= \epsilon$ for $i\geq n+1$. Let 
\begin{center}
	$W_n=\{v_Aw^*:  w\in Q_0, A\subseteq A_w\}\cup\{v_A : A\subseteq \overline{n}\}$.
\end{center}

\noindent On the other hand, we will define now a set of relations. For this let $W_t$ be the set of all words of the form $u_{i_0}u_{i_1}...u_{i_l}x_{j_1}...x_{j_m}x_{j_{m+1}}$ with the following four properties: \\
(i) $l \in \{0,...,n-2\},$ and $m \in \{0,...,n-3\}$; \\
(ii) $i_0<i_1<\cdot\cdot\cdot<i_l\in \{1,...,n-2\};$ \\
(iii) $j_1>j_2>\cdot\cdot\cdot > j_m>j_{m+1}\in \{1,...,n-2\};$ \\
(iv) if $k\in\{i_0,...,i_{l-1}\} \ (k\in\{j_2,...,j_{m+1}\})$ then $k+1, k+3 \notin \{i_1,...,i_l\} \ (k+1, k+3\notin \{j_1,...,j_m\})$ for all $k\in\{1,...,n-3\}$.   \\
Then we define a sequence $R$ of relations on $X_n^*$ as following: For $i,j \in \{1,...,n\}$ and $k={i+2j-2}$, let \\

\noindent
$
\ (E) \ x_iu_j 
\approx  \begin{cases}
v_1v_2v_{i+3}...v_{j+3}, &  \mbox{if}\ i<j, j-i = 2,3  \\
v_1v_2v_{j+3}...v_{i+3}, &  \mbox{if}\ i>j, i-j = 2,3  \\
v_1v_2v_{j+3}v_{j+4}, &  \mbox{if}\ i>j, i-j = 1  \\
v_1v_2v_{j+2}v_{j+3}, &  \mbox{if}\ i<j, j-i = 1  \\
v_1v_2v_{i+3}, &  \mbox{if}\ i=j  \\
v_1v_2u_{j}x_{i+2}, & \mbox{if}\ i<j, j-i\geq 4 \\
v_1v_2u_{j+2}x_{i}, &  \mbox{if}\ i>j, i-j\geq 4;  
\end{cases}
$

	\begin{enumerate} [label = {$(L{\arabic*})$}]
	\item \label{l2} $u_2u_1\approx u_1u_2\approx x_1x_2\approx x_2x_1\approx u_2^2\approx x_2^2\approx v_1v_2v_3v_4v_5$;    
	\item \label{l3} $u_3u_2\approx x_2x_3\approx v_1v_2v_3v_4v_5v_6$;
	\item \label{l4}  $u_iu_1\approx v_1v_2u_i$ and $x_1x_i\approx v_3v_4x_i, i\geq 3$;
	\item \label{l5} $u_iu_2\approx v_1v_2v_3u_i$ and $x_2x_i\approx v_3v_4v_5x_i, i\geq 4$;
	\item \label{l6} $u_iu_{i-1}\approx v_{i+3}u_{i-3}u_{i-1}$ and $x_{i-1}x_i\approx v_{i+3}x_{i-1}x_{i-3}, i\geq 4$;
	\item \label{l7} $u_iu_j\approx u_{j-2}u_i$ and $x_jx_i\approx x_{i}x_{j-2}, i>j\geq 3, i-j\geq2$;
\end{enumerate}
\begin{enumerate}   [label = {$(R{\arabic*})$}]
	\item \label{a} $v_i^2\approx   v_i$, $i\in\{1,...,n\}$;
	\item \label{b} $v_iv_j\approx v_jv_i$, $i,j\in\{1,...,n\}, i\neq j$;
	\item \label{xx} $v_iu_j\approx u_jv_i$ and $v_ix_j\approx x_jv_i$, $i\in\{j+4,...,n\}$;
	\item \label{vv} $v_iu_j\approx u_jv_{i+2}$ and $v_{i+2}x_j\approx x_jv_i$, $1\leq i\leq j$;
	\item \label{uu} $v_iu_j\approx u_j$ and $x_jv_i\approx x_j$, $i\in\{j+1,j+2,j+3\}$;
	\item \label{rr} $u_jv_i\approx u_j$ and $v_ix_j\approx x_j$, $i\in\{1,2,j+3\}$; 		
	\item \label{dd} $u_1^2\approx x_1^2\approx v_1...v_4 $;   
	\item \label{f} $u_i^2\approx u_{i-2}u_i$ and $x_i^2\approx x_ix_{i-2}$, $i\geq 3$; 
	\item \label{c} $u_iu_{i+1}\approx u_{i-1}u_{i+1}$ and $x_{i+1}x_i\approx x_{i+1}x_{i-1}$, $i\in\{2,...,n-5\}$;   
	\item \label{h} $u_iu_{i+3}\approx v_{i+6}u_iu_{i+2}$ and $x_{i+3}x_i\approx v_{i+6}x_{i+2}x_i$, $i\leq n-5$;   
	\item \label{j} $w\approx v_{i_0+1}v_{i_0+2}v_{i_0+3}u_{i_1}...u_{i_l}x_{j_1}...x_{j_m}$,  $w=u_{i_0}u_{i_1}...u_{i_l}x_{j_1}...x_{j_m}x_{j_{m+1}}\in W_t$ with $j_{m+1}=i_0+2l-2m$;  
	\item \label{k} $w\approx v_{i_0}v_{i_0+1}v_{i_0+2}v_{i_0+3}u_{i_1}...u_{i_l}x_{j_1}...x_{j_m}$,  $w=u_{i_0}u_{i_1}...u_{i_l}x_{j_1}...x_{j_m}x_{j_{m+1}}\in W_t$ with $j_{m+1}=i_0+2l-2m-1$;   
	\item \label {l} $w\approx v_{i_0+1}v_{i_0+2}v_{i_0+3}v_{i_0+4}u_{i_1}...u_{i_l}x_{j_1}...x_{j_m}$,    $w=u_{i_0}u_{i_1}...u_{i_l}x_{j_1}...x_{j_m}x_{j_{m+1}}\in W_t$ with $j_{m+1}=i_0+2l-2m+1$;  
	\item \label{m} $w\approx u_{i_0}u_{i_1}...u_{i_l}x_{j_1}...x_{j_m}$, $w=u_{i_0}u_{i_1}...u_{i_l}x_{j_1}...x_{j_m}x_{j_{m+1}}\in W_t$ with $j_{m+1}<2l-2m$; 
	\item \label{mm} $w \approx u_{i_1}...u_{i_l}x_{j_1}...x_{j_m}x_{j_{m+1}}$, $w=u_{i_0}u_{i_1}...u_{i_l}x_{j_1}...x_{j_m}x_{j_{m+1}}\in W_t$ with $i_0<2m-2l$;	 
	\item \label{w} $v_1...v_iu_{i,j}\approx v_1...v_{k+3}$, $i\in\{1,...,n-2\}$;
	\item \label{x} $v_{k-i+3}...v_{k+2}x^{-1}_{i,j}\approx v_1...v_{k+3}$, $i\in\{1,...,n-2\}$;
	\item \label{zz} $v_iu_{i,j}\approx v_{k+3}u_{i-1,j}$, $i\in\{2,...,n-2\}$;
	\item \label{yy} $v_{k+2}x^{-1}_{i,j}\approx v_{k+3}x^{-1}_{{i-1},j}$, $i\in\{2,...,n-2\}$;
\end{enumerate}

\begin{lem} \label{1}  The relations from $R$ hold as equations in $IOF_n^{par}$, when the variables are replaced by the corresponding transformations.
\end{lem}
\begin{proof}
	We show the statement, diagrammatically. We give example calculation, for the relation \ref{h} $u_iu_{i+3}\approx v_{i+6}u_iu_{i+2}$, $i\leq n-5$, in Figure \ref{Fig 1} and \ref{Fig 2} below. Note we can show $x_{i+3}x_i\approx v_{i+6}x_{i+2}x_i$ in a similar way.
\end{proof}
\begin{figure} [H]
	\centering
	\begin{center}
		
		\begin{tikzpicture} [font=\tiny,  xshift=-1cm , bul/.style = {fill=black,circle,inner sep=0.8pt}]
		\node [anchor=east] at (-0.5,1.05) {$\overline{u}_i$};
		\node [anchor=east] at (-0.5,0.35) {$\overline{u}_{i+3}$};
		
		\draw[bul] (0,0.7) circle (1.5pt);
		\draw[bul] (0,1.4) circle (1.5pt);
		
		\draw[bul] (0.5,0) circle (1.5pt);
		\draw[bul] (0.5,0.7) circle (1.5pt);
		
		\draw[bul] (0.875,1.4) circle (0.4pt);
		\draw[bul] (0.675,1.4) circle (0.4pt);
		\draw[bul] (1.075,1.4) circle (0.4pt);
		
		\draw[bul] (1.05,0.7) circle (0.4pt);
		\draw[bul] (1.25,0.7) circle (0.4pt);
		\draw[bul] (1.45,0.7) circle (0.4pt);
		
		\draw[bul] (1.425,0) circle (0.4pt);
		\draw[bul] (1.625,0) circle (0.4pt);
		\draw[bul] (1.825,0) circle (0.4pt);

		\draw[bul] (1.5,1.4) circle (1.5pt);

		\draw[bul] (1.75,1.4) circle (1.5pt);

		\draw[bul] (2,0.7) circle (1.5pt);
		\draw[bul] (2,1.4) circle (1.5pt);

		\draw[bul] (2.25,0.7) circle (1.5pt);
		\draw[bul] (2.25,1.4) circle (1.5pt);

		\draw[bul] (2.5,0.7) circle (1.5pt);
		\draw[bul] (2.5,1.4) circle (1.5pt);
		
		\draw[bul] (2.75,0) circle (1.5pt);
		\draw[bul] (2.75,0.7) circle (1.5pt);

		\draw[bul] (3,0) circle (1.5pt);
		\draw[bul] (3,0.7) circle (1.5pt);

		\draw[bul] (3.25,0) circle (1.5pt);
		\draw[bul] (3.25,0.7) circle (1.5pt);

		\draw[bul] (3.55,0.7) circle (0.4pt);
		\draw[bul] (3.75,0.7) circle (0.4pt);
		\draw[bul] (3.95,0.7) circle (0.4pt);
		
		\draw[bul] (3.55,0) circle (0.4pt);
		\draw[bul] (3.75,0) circle (0.4pt);
		\draw[bul] (3.95,0) circle (0.4pt);
		
		\draw[bul] (3.175,1.4) circle (0.4pt);
		\draw[bul] (3.375,1.4) circle (0.4pt);
		\draw[bul] (3.575,1.4) circle (0.4pt);
		
		\draw[bul] (4.25,0) circle (1.5pt);
		\draw[bul] (4.25,0.7) circle (1.5pt);
		\draw[bul] (4.25,1.4) circle (1.5pt);
		
		\draw[bul] (5.8,1.4) circle (1.5pt);
		\draw[bul] (6.8,0) circle (1.5pt);
		\draw (5.8,1.4)--(6.8,0);
		\draw[bul] (7.3,1.4) circle (1.5pt);
		\draw[bul] (8.3,0) circle (1.5pt);
		\draw (7.3,1.4)--(8.3,0);
		\draw[bul] (9.05,1.4) circle (1.5pt);
		\draw[bul] (9.05,0) circle (1.5pt);
		\draw (9.05,1.4)--(9.05,0);
		
		\draw[bul] (7.55,1.4) circle (1.5pt);
		\draw[bul] (7.8,1.4) circle (1.5pt);
		\draw[bul] (8.05,1.4) circle (1.5pt);
		\draw[bul] (8.3,1.4) circle (1.5pt);
		\draw[bul] (8.55,1.4) circle (1.5pt);
		\draw[bul] (8.8,1.4) circle (1.5pt);
		
		\draw[bul] (8.55,0) circle (1.5pt);
		\draw[bul] (8.8,0) circle (1.5pt);
		
		\draw[bul] (10.05,1.4) circle (1.5pt);
		\draw[bul] 	(10.05,0) circle (1.5pt);
		\draw (10.05,1.4)--(10.05,0);
		
		\draw[bul] (6.35,1.4) circle (0.4pt);
		\draw[bul] (6.55,1.4) circle (0.4pt);
		\draw[bul] (6.75,1.4) circle (0.4pt);

		\draw[bul] (7.35,0) circle (0.4pt);
		\draw[bul] (7.55,0) circle (0.4pt);
		\draw[bul] (7.75,0) circle (0.4pt);
		
		\draw[bul] (9.4,0) circle (0.4pt);
		\draw[bul] (9.6,0) circle (0.4pt);
		\draw[bul] (9.8,0) circle (0.4pt);
		\draw[bul] (9.4,1.4) circle (0.4pt);
		\draw[bul] (9.6,1.4) circle (0.4pt);
		\draw[bul] (9.8,1.4) circle (0.4pt);
		
		\node [anchor=west] at (10.4,0.7) {$\overline{u}_i\overline{u}_{i+3}$};
		
		\node [anchor=south] at (5.8,1.4) {$1$};
		\node [anchor=south] at (7.3,1.4) {$i$};
		\node [anchor=south] at (9.05,1.4) {$i+7$};
		\node [anchor=south] at (10.05,1.4) {$n$};
		\node [anchor=north] at (6.8,0) {$5$};
		\node [anchor=north] at (8.3,0) {$i+4$};
		
		\node [anchor=south] at (0,1.4) {$1$};
		\node [anchor=south] at (0.5,0.7) {$3$};
		\node [anchor=south] at (1.5,1.4) {$i$};
		\node [anchor=south] at (2.5,1.4) {$i+4$};
		\node [anchor=west] at (4.75,0.7) {$=$};
		\node [anchor=south] at (4.25,1.4) {$n$};
		\node [anchor=south] at (3.25,0.7) {$i+7$};		
		\draw (0,1.4)--(0.5,0.7) (0,0.7)--(0.5,0) (1.5,1.4)--(2,0.7) (2.25,0.7)--(2.75,0) (2.5,1.4)--(2.5,0.7) (3.25,0.7)--(3.25,0) (4.25,1.4)--(4.25,0);
		\end{tikzpicture} 	
	\end{center}
	\caption{ $\overline{u}_i\overline{u}_{i+3}$.} \label{Fig 1} 
\end{figure}
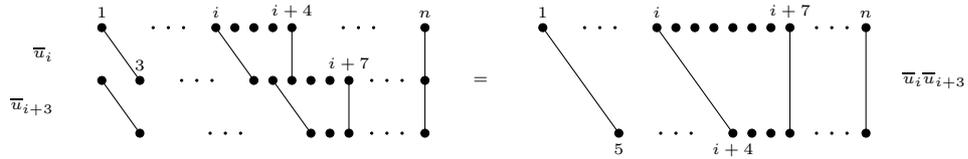

\begin{figure} [H]
	\centering
	\begin{center}
		
		\begin{tikzpicture} [font=\tiny,  xshift=-1cm , bul/.style = {fill=black,circle,inner sep=0.8pt}]
		
		\node [anchor=east] at (-0.5,1.05) {$\overline{u}_{i}$};
		\node [anchor=east] at (-0.5,0.35) {$\overline{u}_{i+2}$};
		\node [anchor=east] at (-0.5,1.75) {$\overline{v}_{i+6}$};
		
		\draw[bul] (0,2.1) circle (1.5pt);
		
		\draw[bul] (2.75,2.1) circle (1.5pt);
		\draw[bul] (3,2.1) circle (1.5pt);
		\draw[bul] (3.25,2.1) circle (1.5pt);
		
		\draw[bul] (1.175,2.1) circle (0.4pt);
		\draw[bul] (1.375,2.1) circle (0.4pt);
		\draw[bul] (1.575,2.1) circle (0.4pt);
		
		\draw[bul] (3.55,2.1) circle (0.4pt);
		\draw[bul] (3.75,2.1) circle (0.4pt);
		\draw[bul] (3.95,2.1) circle (0.4pt);
		\draw (0,2.1)--(0,1.4) (3.25,2.1)--(3.25,1.4) (2.75,2.1)--(2.75,1.4) (4.25,2.1)--(4.25,1.4);
		\draw[bul] (4.25,2.1) circle (1.5pt);

		\draw[bul] (0,0.7) circle (1.5pt);
		\draw[bul] (0,1.4) circle (1.5pt);

		\draw[bul] (0.5,0) circle (1.5pt);
		\draw[bul] (0.5,0.7) circle (1.5pt);

		\draw[bul] (0.875,1.4) circle (0.4pt);
		\draw[bul] (0.675,1.4) circle (0.4pt);
		\draw[bul] (1.075,1.4) circle (0.4pt);

		\draw[bul] (1.05,0.7) circle (0.4pt);
		\draw[bul] (1.25,0.7) circle (0.4pt);
		\draw[bul] (1.45,0.7) circle (0.4pt);
		
		\draw[bul] (1.05,0) circle (0.4pt);
		\draw[bul] (1.25,0) circle (0.4pt);
		\draw[bul] (1.45,0) circle (0.4pt);
		
		\draw[bul] (1.5,1.4) circle (1.5pt);

		\draw[bul] (1.75,1.4) circle (1.5pt);
		
		\draw[bul] (2,0.7) circle (1.5pt);
		\draw[bul] (2,1.4) circle (1.5pt);

		\draw[bul] (2.25,0.7) circle (1.5pt);
		\draw[bul] (2.25,1.4) circle (1.5pt);
		
		\draw[bul] (2.5,0) circle (1.5pt);
		\draw[bul] (2.5,0.7) circle (1.5pt);
		\draw[bul] (2.5,1.4) circle (1.5pt);
		
		\draw[bul] (2.75,0) circle (1.5pt);
		\draw[bul] (2.75,0.7) circle (1.5pt);
		\draw[bul] (2.75,1.4) circle (1.5pt);
		
		\draw[bul] (3,0) circle (1.5pt);
		\draw[bul] (3,0.7) circle (1.5pt);
		\draw[bul] (3,1.4) circle (1.5pt);
		
		\draw[bul] (3.25,0) circle (1.5pt);
		\draw[bul] (3.25,0.7) circle (1.5pt);
		\draw[bul] (3.25,1.4) circle (1.5pt);
		
		\draw[bul] (3.55,0) circle (0.4pt);
		\draw[bul] (3.75,0) circle (0.4pt);
		\draw[bul] (3.95,0) circle (0.4pt);
		
		\draw[bul] (3.55,0.7) circle (0.4pt);
		\draw[bul] (3.75,0.7) circle (0.4pt);
		\draw[bul] (3.95,0.7) circle (0.4pt);
		
		\draw[bul] (3.55,1.4) circle (0.4pt);
		\draw[bul] (3.75,1.4) circle (0.4pt);
		\draw[bul] (3.95,1.4) circle (0.4pt);
		
		\draw[bul] (4.25,0) circle (1.5pt);
		\draw[bul] (4.25,0.7) circle (1.5pt);
		\draw[bul] (4.25,1.4) circle (1.5pt);
		
		\draw[bul] (5.8,2.1) circle (1.5pt);
		\draw[bul] (6.8,0) circle (1.5pt);
		\draw (5.8,2.1)--(6.8,0);
		\draw[bul] (7.3,2.1) circle (1.5pt);
		\draw[bul] (8.3,0) circle (1.5pt);
		\draw (7.3,2.1)--(8.3,0);
		\draw[bul] (9.05,2.1) circle (1.5pt);
		\draw[bul] (9.05,0) circle (1.5pt);
		\draw (9.05,2.1)--(9.05,0);
		
		\draw[bul] (7.55,2.1) circle (1.5pt);
		\draw[bul] (7.8,2.1) circle (1.5pt);
		\draw[bul] (8.05,2.1) circle (1.5pt);
		\draw[bul] (8.3,2.1) circle (1.5pt);
		\draw[bul] (8.55,2.1) circle (1.5pt);
		\draw[bul] (8.8,2.1) circle (1.5pt);
		
		\draw[bul] (8.55,0) circle (1.5pt);
		\draw[bul] (8.8,0) circle (1.5pt);
		
		\draw[bul] (10.05,2.1) circle (1.5pt);
		\draw[bul] 	(10.05,0) circle (1.5pt);
		\draw (10.05,2.1)--(10.05,0);
		
		\draw[bul] (6.35,2.1) circle (0.4pt);
		\draw[bul] (6.55,2.1) circle (0.4pt);
		\draw[bul] (6.75,2.1) circle (0.4pt);

		\draw[bul] (7.35,0) circle (0.4pt);
		\draw[bul] (7.55,0) circle (0.4pt);
		\draw[bul] (7.75,0) circle (0.4pt);
		
		\draw[bul] (9.4,0) circle (0.4pt);
		\draw[bul] (9.6,0) circle (0.4pt);
		\draw[bul] (9.8,0) circle (0.4pt);
		\draw[bul] (9.4,2.1) circle (0.4pt);
		\draw[bul] (9.6,2.1) circle (0.4pt);
		\draw[bul] (9.8,2.1) circle (0.4pt);

		\node [anchor=south] at (5.8,2.1) {$1$};
		\node [anchor=south] at (7.3,2.1) {$i$};
		\node [anchor=south] at (9.05,2.1) {$i+7$};
		\node [anchor=south] at (10.05,2.1) {$n$};
		\node [anchor=north] at (6.8,0) {$5$};
		\node [anchor=north] at (8.3,0) {$i+4$};
		
		

		\node [anchor=west] at (10.4,1.1) {$\overline{v}_{i+6}\overline{u}_i\overline{u}_{i+2}$};

		\node [anchor=south] at (0,2.1) {$1$};
		\node [anchor=south] at (0.5,0.7) {$3$};
		\node [anchor=south] at (1.5,1.4) {$i$};
		\node [anchor=south] at (3.25,2.1) {$i+7$};
		\node [anchor=west] at (4.75,1.0) {$=$};
		\node [anchor=south] at (4.25,2.1) {$n$};
		
		\draw (0,1.4)--(0.5,0.7) (0,0.7)--(0.5,0) (1.5,1.4)--(2.5,0) (2.5,1.4)--(2.5,0.7) (2.75,1.4)--(2.75,0.7) (3,1.4)--(3,0.7) (3.25,1.4)--(3.25,0.7) (3,0.7)--(3,0) (3.25,0.7)--(3.25,0) (4.25,1.4)--(4.25,0) ; 
		\end{tikzpicture} 
	\end{center} 
	\caption{ $\overline{v}_{i+6}\overline{u}_i\overline{u}_{i+2}$.} \label{Fig 2} 
	\begin{center}
		By Figure \ref{Fig 1} and \ref{Fig 2}, we have that $\overline{u}_i\overline{u}_{i+3}= \overline{v}_{i+6}\overline{u}_i\overline{u}_{i+2}$. 
	\end{center}	
\end{figure} 
 Next, we will verify  consequences of $R$, which are important by technical reasons.
\begin{lem} \label{i-ii} 
	(i) For  $w=u_{i_0}u_{i_1}...u_{i_l}x_{j_1}...x_{j_m}x_{j_{m+1}}\in W_t$ with $j_{m+1}=2l-2m$, we have\\  $w\approx v_1u_{i_0}u_{i_1}...u_{i_l}x_{j_1}...x_{j_m}$.  \\
\indent \quad \quad \ \ \ \	(ii) For  $w=u_{i_0}u_{i_1}...u_{i_l}x_{j_1}...x_{j_m}x_{j_{m+1}}\in W_t$ with $i_0=2m-2l$, we have \\ $w\approx v_{i_0+3}u_{i_1}...u_{i_l}x_{j_1}...x_{j_m}x_{j_{m+1}}$.
\end{lem}
\begin{proof}
	(i) We have  $u_{i_0}u_{i_1}...u_{i_l}x_{j_1}...x_{j_m}x_{j_{m+1}}\stackrel{\text{\ref{m}}}{\approx} u_{i_0}u_{i_1}...u_{i_l}x_{j_1}...x_{j_m}x_{j_{m+1}-1}x_{j_{m+1}}$. \\
	Suppose   $j_{m+1}=2l-2m\geq 4$.  Then $u_{i_0}u_{i_1}...u_{i_l}x_{j_1}...x_{j_m}x_{j_{m+1}-1}x_{j_{m+1}}\stackrel{\text{\ref{l6}}}{\approx} \\ u_{i_0}u_{i_1}...u_{i_l}x_{j_1}...x_{j_m}v_{{j_{m+1}+3}}x_{j_{m+1}-1}x_{{j_{m+1}-3}} \stackrel{\text{\ref{vv}}}{\approx}  v_1u_{i_0}u_{i_1}...u_{i_l}x_{j_1}...x_{j_m}x_{{j_{m+1}-1}}x_{{j_{m+1}-3}} \\ \stackrel{\text{\ref{m}}}{\approx}  v_1u_{i_0}u_{i_1}...u_{i_l}x_{j_1}...x_{j_m} $. Suppose $j_{m+1}=2l-2m< 4$, i.e. $j_{m+1}=2$. We prove in a similar way that $u_{i_0}u_{i_1}...u_{i_l}x_{j_1}...x_{j_m}x_{j_{m+1}}\approx v_1u_{i_0}u_{i_1}...u_{i_l}x_{j_1}...x_{j_m}$ by using  \ref{l2}	and \ref{vv}-\ref{rr}. \\
	
	\noindent (ii) The proof is similar to (i), by using \ref{mm} and \ref{l6} if $i_0\geq 4$ and  \ref{mm}, \ref{l2}, and \ref{vv}-\ref{rr}  if $i_0=2$.
\end{proof}
	\section{Set of Forms}
In this section, we introduce an algorithm, which transforms any word $w\in\ X_n^*$ to a word in $W_n$ using $R$, with other words, we show that for all $w\in\ X_n^*$, there is $w'\in W_n$ such that $w\approx w'$ is a consequence of $R$. First, the algorithm transforms each $w\in\ X_n^*$ to a "new" word $w'$. All these "new" words will be collected in a set. Later, we show that set belongs to $W_n$.  Let $w\in X_n^*\backslash \{\epsilon\}$.
\begin{itemize}
	\item	Using \ref{a}-\ref{rr}, we move in the front of $\tilde{w}$ and cancel all ${v_j}'s$, for $j\in\{1,...,n\}$, respectively.      We have $w\approx \tilde{v}\tilde{w}$, where $\tilde{v}\in\{v_1,...,v_n\}^*$ and $\tilde{w}\in\{u_1,u_2,...,u_{n-2},x_1,x_2\\ ,...,x_{n-2}\}^*$. 
	\item Moreover, we seperate the ${u_i} 's$ and ${x_i}'s$ for $i\in\{1,...,n-2\}$ by $(E)$ and \ref{a}-\ref{rr}. Then $\tilde{v}\tilde{w} \approx \overline{v}\overline{B}\overline{C}$, where $\overline{v}\in\{v_1,...,v_n\}^*$, $\overline{B}\in\{u_1,u_2,...,u_{n-2}\}^*$, and $\overline{C}\in\{x_1,x_2,...,x_{n-2}\}^*$. 
	\item By \ref{l2}-\ref{l7} and \ref{a}-\ref{rr}, we get $\overline{v}\overline{B}\overline{C}\approx v'B'C'$, where $v'\in\{v_1,...,v_n\}^*$, $B'\in\{u_1,u_2,...,u_{n-2}\}^*$, and $C'\in\{x_1,x_2,...,x_{n-2}\}^*$ such that the indexes of the variables in the word $B'$ are ascending and in the word $C'$ are descending (reading from the left to the right).
	\item By \ref{l2},  \ref{dd}-\ref{h}, and \ref{a}-\ref{rr}, we replace subwords of $B'C'$ of the form $x_{i+3}x_{i}, x_{i+1}x_i, \\ x_i^2, u_i^2, u_iu_{i+3},$ and $u_iu_{i+1}$ until $v'B'C'\approx v''w_1...w_p$ with $v''\in\{v_1,...,v_n\}^*$ and ${w_1},...,{w}_{p} \in W_x^{-1}\cup W_u$  such that 
	\begin{multline} \label{(1)}
	\mbox{if} \ u_i\in var(w_1...w_p)\ (x_i\in var(w_1...w_p)) \ \mbox{then}  \ u_{i+1}, u_{i+3} \notin var(w_1...w_p)\  (x_{i+1}, x_{i+3} \notin \\ \quad var(w_1...w_p)) \ \mbox{for all} \ i\in\{1,...,n-2\} \ \mbox{and each variable in} \ w_1...w_p \ \mbox{is unique}.  \hfill
	\end{multline}
	Note that this    is possible since each of the relations \ref{l2},  \ref{dd}-\ref{h}, and \ref{a}-\ref{rr} do not increase the index of any variable in  $\{u_1,u_2,...,u_{n-2},x_1,x_2,...,x_{n-2}\}$ in the "new" word.  
	\item  Using \ref{j}-\ref{mm},  Lemma \ref{i-ii}, and \ref{a}-\ref{rr}, we remove variables $x_i$ and $u_i$, respectively,  until one can not more remove a variable $x_i$ or $u_i$. We obtain $v''w_1...w_p\approx {v'''}{w'_1}...{w'}_{{p'}}$, where $v'''\in\{v_1,...,v_n\}^*$ and ${w'_1},...,{w'}_{p'} \in W_x^{-1}\cup W_u$. Note that is possible since each of the relations  \ref{j}-\ref{mm} as well as  Lemma \ref{i-ii} only remove variables (and add variables in $\{v_1,...,v_n\}$, respectively).
	\item We decrease the indexes of the variable in $\{u_1,u_2,...,u_{n-2},x_1,x_2,...,x_{n-2}\}$ (if possible) by \ref{w}-\ref{yy} as well as \ref{a}-\ref{rr} and obtain ${v'''}{w'_1}...{w'}_{{p'}}\approx v^*B^*C^*$ with $v^*\in\{v_1,...,v_n\}^*$, ${B}^*\in\{u_1,u_2,...,u_{n-2}\}^*$, and ${C}^*\in\{x_1,x_2,...,x_{n-2}\}^*$. Note that the indexes of the variables in $B^*$ (in $C^*$) are ascending (are descending).  
\end{itemize}
\noindent We repeat all steps. The procedure terminates if in all steps, the word will not more changed. We obtain  $v^*B^*C^*\approx v_A\hat{w_1}...\hat{w}_{\hat{p}}$, where $\hat{w_1},...,\hat{w}_{\hat{p}}\in W_x^{-1}\cup W_u$ and  $A\subseteq \overline{n}$ such that no $v_j$ ($j\in A$) can be canceled by using \ref{a}-\ref{rr}. This case has to happen since in each step the number of the variables from $\{u_1,u_2,...,u_{n-2},x_1,x_2,...,x_{n-2},v_1,...,v_n\}$ decreases or is kept and the indexes of the ${u_i}'s$ and ${x_i}'s$ decrease or are kept. \\

\noindent We give the set of all words, which we obtain from words in $w\in\ X_n^*$ by that  algorithm, the name $P$. \\

 By (\ref{(1)}), we obtain immediately from algorithm:

\begin{remark} \label{two}
	Let $\hat{w}=v_A\hat{w}_1...\hat{w}_{m}\in P$ and let $1\leq k<k'\leq m$. \\
	If $\hat{w}_k,\hat{w}_{k'}\in W_u$ then $i_k+2\lvert \hat{w}_k \rvert +2\leq i_{k'}$. \\         
	If $\hat{w}_k,\hat{w}_{k'}\in W_x$ then $i_{k'}+2\lvert \hat{w}_{k'} \rvert +2\leq i_k$.
\end{remark}

Let fix a word $\hat{w}=v_A\hat{w}_1...\hat{w}_{m}\in P$. There are $a,b\in\{0,...,n\}$ with $a+b=m$, $t_1,...,t_{a+b}\in\{1,...,m\}$, $w_{t_1},...,w_{t_a}\in W_u$ and $w_{t_{a+1}},...,w_{t_{a+b}}\in W_x$, such that $\hat{w}=v_A\hat{w}_1...\hat{w}_{m}=v_Aw_{t_1}...w_{t_a}w_{t+1}^{-1}\\...w_{t_{a+b}}^{-1}$, where $\{w_{t_1},...,w_{t_a}\}=\emptyset$ or $\{w_{t_{a+1}},...,w_{t_{a+b}}\}=\emptyset$ (i.e. $a=0$ or $b=0$) is possible. We  observe that    $\{\hat{w}_1,...,\hat{w}_m\}=\{w_{t_1},...,w_{t_a},w_{t+1}^{-1},...,w_{t_{a+b}}^{-1}\}$ and $\{t_1,...,t_a,t_{a+1},...,t_{a+b}\}=\{1,...,m\}$. We define an order on $\{t_1,...,t_a,t_{a+1},...,t_{a+b}\}$  by $t_1<\cdot\cdot\cdot<t_a$ and $t_{a+b}<\cdot\cdot\cdot<t_{a+1}$. If $a,b\geq 1$, the order between $t_1,...,t_a$ and $t_{a+1},...,t_{a+b}$ is given by the following rule:	Let $k\in\{1,...,a\}$ and $l\in\{1,...,b\}$. 
\begin{center}
	If $i_{t_k}+2\lvert w_{t_k} \rvert-2+2\lvert w_{t_{k+1}}...w_{t_{a}}\rvert-2\lvert w_{t_{a+1}}^{-1}...w_{t_{a+l-1}}^{-1}\rvert< i_{t_{a+l}}+2\lvert w_{t_{a+l}}^{-1}\rvert-2$ then $t_k<t_{a+l}$
\end{center} 
and if $i_{t_k}+2\lvert w_{t_k} \rvert-2+2\lvert w_{t_{k+1}}...w_{t_{a}}\rvert-2\lvert w_{t_{a+1}}^{-1}...w_{t_{a+l-1}}^{-1}\rvert> i_{t_{a+l}}+2\lvert w_{t_{a+l}}^{-1}\rvert-2$ then $t_k>t_{a+l}$. \\

\noindent  The case $i_{t_k}+2\lvert w_{t_k} \rvert-2+2\lvert w_{t_{k+1}}...w_{t_{a}}\rvert-2\lvert w_{t_{a+1}}^{-1}...w_{t_{a+l-1}}^{-1}\rvert = i_{t_{a+l}}+2\lvert w_{t_{a+l}}^{-1}\rvert-2$ is not possible, since otherwise we can cancel $u_{i_{t_k}+2\lvert w_{t_k} \rvert-2}$ and $x_{i_{t_{a+l}}+2\lvert w_{t_{a+l}}^{-1}\rvert-2}$ in $\hat{w}$ by \ref{j}. \\

Our next aim is to describe the relationships between $k_u, (k+1)_u$ and $k_x,(k+1)_x$  for all $k\in\{1,...,m-1\}$ for the word $w=w_1...w_m$. 
\begin{lem} \label{k<k+1} For all $k\in\{1,...,m-1\}$, we have $k_u<(k+1)_u$ and $k_x<(k+1)_x$. \end{lem}
\begin{proof} Let $k\in\{1,...,m-1\}$. 
	Suppose $w_k, w_{k+1}\in W_u$. We obtain $k_u<(k+1)_u$ and $k_x=i_k+2\lvert w_k \rvert +2\lvert W_u^k\rvert - 2\lvert W_x^k \rvert$, $(k+1)_x=i_{k+1}+2\lvert w_{k+1} \rvert +2\lvert W_u^{k+1}\rvert - 2\lvert W_x^{k+1}\rvert$. By Remark \ref{two}, we have $i_k+2\lvert w_k \rvert +2\leq i_{k+1}$. This gives, $i_k+2\lvert w_k \rvert +2\lvert W_u^k\rvert - 2\lvert W_x^k \rvert < i_{k+1}+ 2\lvert W_u^k\rvert - 2\lvert W_x^k \rvert= i_{k+1}+2\lvert w_{k+1}\rvert+ 2\lvert W_u^{k+1}\rvert - 2\lvert W_x^{k+1} \rvert$ (since $w_{k+1}\in W_u$ implies $2\lvert W_x^{k}\rvert=2\lvert W_x^{k+1}\rvert$ ). Then $k_x<(k+1)_x$. For the case $w_k,w_{k+1}\in W_x$, we can show that $k_u<(k+1)_u$ and $k_x<(k+1)_x$ in a similar way. \\   
	Suppose $w_k\in W_u$ and $w_{k+1}\in W_x$. 
	First, we will show $k_u<(k+1)_u$. We have $k_u=i_k$ and $(k+1)_u=i_{k+1}+2\lvert w_{k+1}\rvert+2\lvert W_x^{k+1}\rvert -2\lvert W_u^{k+1}\rvert$. Since $k\in\{t_1,...,t_a\}$ and $k+1\in\{t_{a+1},...,t_{a+b}\}$, we obtain $i_{k}+2\lvert w_{k}\rvert-2+2\lvert W_u^{k}\rvert -2\lvert W_x^{k+1}\rvert < i_{k+1}+2\lvert w_{k+1}\rvert-2$. Then $i_k<i_k+2\lvert w_k\rvert < i_{k+1}+2\lvert w_{k+1}\rvert+2\lvert W_x^{k+1}\rvert -2\lvert W_u^{k+1}\rvert$ (since $w_{k+1}\in W_x$ implies $\lvert W_u^{k}\rvert=\lvert W_u^{k+1}\rvert$ ). Then $k_u<(k+1)_u$. Moreover, we prove $k_x<(k+1)_x$ similarly. The case  $w_k\in W_x$ and $w_{k+1}\in W_u$ can be shown in a similar way as above.
\end{proof}
Of course, the next goal should  be the proof of  $w=w_1...w_{m}\in Q_0$, i.e. we will show that $w$ satisfies \ref{1W}-\ref{4W}. 
\begin{lem} \label{w in Q}
	We have $w=w_1...w_{m}\in Q_0$.
\end{lem}
\begin{proof}
	Exactly, $w$ satisfies \ref{1W} and \ref{2W}. These are trivially checked by Remark \ref{two}. \\
	Let $k\in\{1,...,m-1\}$ and let $w_k\in W_u, w_{k+1}\in W_x $. This provides $k\in\{t_1,...,t_a\}, k+1\in\{t_{a+1},...,t_{a+b}\}$. We have $i_k+2\lvert w_{k}\rvert-2+2\lvert W_u^{k}\rvert -2\lvert W_x^{k+1}\rvert< i_{k+1}+2\lvert w_{k+1}\rvert-2$. Since $w_{k+1}\in W_x$, we have $2\lvert W_u^{k}\rvert=2\lvert W_u^{k+1}\rvert$. So $i_k+2\lvert w_{k}\rvert-2+2\lvert W_u^{k+1}\rvert -2\lvert W_x^{k+1}\rvert< i_{k+1}+2\lvert w_{k+1}\rvert-2$. We observe that $i_k+2\lvert w_{k}\rvert-2+2\lvert W_u^{k+1}\rvert -2\lvert W_x^{k+1}\rvert+1 \leq i_{k+1}+2\lvert w_{k+1}\rvert-2$. If $i_k+2\lvert w_{k}\rvert-2+2\lvert W_u^{k+1}\rvert -2\lvert W_x^{k+1}\rvert+1= i_{k+1}+2\lvert w_{k+1}\rvert-2$, we can cancel $u_{i_k+2\lvert w_{k}\rvert-2} , x_{i_{k+1}+2\lvert w_{k+1}\rvert-2}$ by \ref{l} in $\hat{w}$. This contradicts $\hat{w}\in P$. Then  $i_k+2\lvert w_{k}\rvert-2+2\lvert W_u^{k+1}\rvert -2\lvert W_x^{k+1}\rvert+2 \leq i_{k+1}+2\lvert w_{k+1}\rvert-2$, i.e.  $i_k+2\lvert w_{k}\rvert+2 \leq i_{k+1}+2\lvert w_{k+1}\rvert-2\lvert W_u^{k+1}\rvert +2\lvert W_x^{k+1}\rvert=(k+1)_u$. Next, to show that $(k+1)_x-k_x\geq 2$. Lemma \ref{k<k+1} gives $(k+1)_x-k_x\geq1$. If $(k+1)_x-k_x=1$ then $i_{k+1}-i_k-2\lvert w_{k}\rvert-2\lvert W_u^{k}\rvert +2\lvert W_x^{k}\rvert=1$. This implies  $i_{k+1}+2\lvert w_{k+1}\rvert-2=i_k+2\lvert w_{k}\rvert-2+2\lvert W_u^{k}\rvert -2\lvert W_x^{k+1}\rvert+1$ since $2\lvert W_x^{k}\rvert=2\lvert W_{k+1}\rvert+2\lvert W_x^{k+1}\rvert$. We can cancel $u_{i_k+2\lvert w_{k}\rvert-2} , x_{i_{k+1}+2\lvert w_{k+1}\rvert-2}$ by \ref{l} in $\hat{w}$. This contradicts $\hat{w}\in P$. Thus, $(k+1)_x-k_x\geq2$. In case $w_k,w_{k+1}\in W_u$, by using Remark \ref{two}, we easily get $i_k+2\lvert w_{k}\rvert+2\leq (k+1)_u$. For show $(k+1)_x-k_x\geq 2$, it is routine to calculate directly. Together with Remark \ref{two}, we will get that  $(k+1)_x-k_x\geq 2$. Altogether, $w$ satisfies \ref{3W}. We prove that $w$ satisfies \ref{4W} in a similar way. Therefore, $w\in Q_0$. 
\end{proof} 

We are now in position and have shown $w\in Q_0$. These are leading us to the next step, showing that $A\subseteq A_w$. First, we point out subsets of $\overline{n}$, which do not contains any element of $A$.

\begin{lem} \label{mid}
	Let $q\in\{1,...,a\}$ and  let $\rho\in\{i_{t_q}+1,...,i_{t_q}+2\lvert w_{t_q}\rvert+1\}\cap\overline{n}$. Then $\rho\notin A$. 
\end{lem}
\begin{proof}
	Assume $\rho\in A$. Then  $v_\rho w_{t_1}...w_{t_q}...w_{t_a}w_{t_{a+1}}^{-1}...w_{t_{a+b}}^{-1}\stackrel{\text{\ref{xx}}}{\approx}w_{t_1}...v_\rho w_{t_q}...w_{t_a}w_{t_{a+1}}^{-1}...w_{t_{a+b}}^{-1}$. \\
	If $\rho\in\{i_{t_q}+1,i_{t_q}+2,i_{t_q}+3\}\cap\overline{n}$ then $v_\rho u_{i_{t_q}}\stackrel{\text{\ref{uu}}}{\approx} u_{i_{t_q}}$. \\
	If $\rho =i_{t_q}+h+t$ for some $h\in\{2,4,...,2\lvert w_{t_q}\rvert-2\}$ and $t\in\{2,3\}$ then \\ $w_{t_1}...v_\rho w_{t_q}...w_{t_a}w_{t_{a+1}}^{-1}...w_{t_{a+b}}^{-1} \\ =w_{t_1}...v_\rho u_{i_{t_q}}u_{i_{t_q}+2}...u_{i_{t_q}+2\lvert w_{t_q} \rvert-2}w_{t_{q+1}}...w_{t_a}w_{t_{a+1}}^{-1}...w_{t_{a+b}}^{-1}$ \\
	$\stackrel{\text{\ref{xx}}}{\approx} w_{t_1}... u_{i_{t_q}}...v_{(i_{t_q}+h+t)}u_{i_{t_q}+h}...u_{i_{t_q}+2\lvert w_{t_q} \rvert-2}w_{t_{q+1}}...w_{t_a}w_{t_{a+1}}^{-1}...w_{t_{a+b}}^{-1}$ \\
	$\stackrel{\text{\ref{uu}}}{\approx} w_{t_1}... u_{i_{t_q}}...u_{i_{t_q}+h}...u_{i_{t_q}+2\lvert w_{t_q} \rvert-2}w_{t_{q+1}}...w_{t_a}w_{t_{a+1}}^{-1}...w_{t_{a+b}}^{-1}$,\\
	i.e. we can cancel $v_\rho$ in $\hat{w}$ using \ref{xx} and \ref{uu}, a contradiction. 
\end{proof}

\begin{lem} \label{4} Let $\rho\in A$ and let $q\in\{1,...,a\}$ such that $t_q\neq m$. If $\rho\in\{({t_q})_u+1,...,(t_q+1)_u-1\}$ then $\rho\in\{({t_q})_u+2\lvert w_{t_q}\rvert+2,...,(t_q+1)_u-1\}\subseteq A_w$.
\end{lem}
\begin{proof}
	We have $({t_q})_u=i_{t_q}$. It is a consequence of Lemma \ref{mid} that  $\rho\in\{i_{t_q}+2\lvert w_{t_q}\rvert+2,...,(t_q+1)_u-1\}$ and by \ref{6w}, we have $\{i_{t_q}+2\lvert w_{t_q}\rvert+2,...,(t_q+1)_u-1\}\subseteq A_w$.
\end{proof}
\begin{lem} \label{5}
	Let $\rho\in A$. If  $t_a=m$ and $\rho\in\{i_m+1,...,n\}$ then $\rho\in\{m_x+2,...,n\}\subseteq A_w$.
\end{lem}
\begin{proof} 
	Assume $\rho\in \{i_{m}+1,...,m_x+1\}$. We have $m_x+1=i_{t_a}+2\lvert w_{t_a} \rvert+1$. Then $\rho\in \{i_{t_a}+1,...,i_{t_a}+2\lvert w_{t_a} \rvert+1\}$. By Lemma \ref{mid}, we have $\rho\notin A$. Therefore, $\rho\in\{m_x+2,...,n\}\subseteq A_w$ by \ref{5W}.
\end{proof}

\begin{lem} \label{x middle}
	Let $\rho\in A$. Then $\rho\neq (t_{a+l})_u+1$ for all $l\in\{1,...,b\}$.
\end{lem}
\begin{proof}
	Let $l\in\{1,...,b\}$. Assume $\rho = (t_{a+l})_u+1$. Suppose there exist $q\in\{1,...,a\}$ with $t_q>t_{a+l}$. Then 	$v_{\rho}w_{t_1}...w_{t_q}...w_{t_a}w_{t_{a+1}}^{-1}...w_{t_{a+b}}^{-1}  \stackrel{\text{\ref{xx}}}{\approx}w_{t_1}...v_\rho w_{t_q}...w_{t_a}w_{t_{a+1}}^{-1}...w_{t_{a+b}}^{-1} \stackrel{\text{\ref{vv}}}{\approx} w_{t_1}...w_{t_q}...w_{t_a}v_{\rho+2\lvert w_{t_q}...w_{t_a}\rvert} \\ w_{t_{a+1}}^{-1}...w_{t_{a+b}}^{-1}$. 
	Since $(t_{a+l})_u+1=i_{t_{a+l}}+2\lvert w_{t_{a+1}}^{-1}...w_{t_{a+l}}^{-1}\rvert-2\lvert w_{t_q}...w_{t_a}\rvert+1$, we have $\rho+2\lvert w_{t_q}...w_{t_a} \rvert=i_{t_{a+l}}+2\lvert w_{t_{a+1}}^{-1}...w_{t_{a+l}}^{-1}\rvert+1$. 
	Suppose  $t_q<t_{a+l}$ for all $q\in\{1,...,a\}$. Then we have $(t_{a+l})_u+1=i_{t_{a+l}}+2\lvert w_{t_{a+1}}^{-1}...w_{t_{a+l}}^{-1}\rvert+1$, i.e.  $v_{\rho}w_{t_1}...w_{t_q}...w_{t_a}w_{t_{a+1}}^{-1}...w_{t_{a+b}}^{-1}\stackrel{\text{\ref{xx}}}{\approx}w_{t_1}... w_{t_q}...w_{t_a}v_\rho w_{t_{a+1}}^{-1}...w_{t_{a+b}}^{-1}$. \\
	Both cases imply  $w_{t_1}...w_{t_q}...w_{t_a}v_{i_{t_{a+l}}+2\lvert w_{t_{a+1}}^{-1}...w_{t_{a+l}}^{-1}\rvert+1}  w_{t_{a+1}}^{-1}...w_{t_{a+b}}^{-1} \stackrel{\text{\ref{vv}}}{\approx} w_{t_1}... w_{t_q}...w_{t_a} w_{t_{a+1}}^{-1}...\\ v_{i_{t_{a+l}}+2\lvert w_{t_{a+l}}^{-1}\rvert+1}w_{t_{a+l}}^{-1}...w_{t_{a+b}}^{-1}$
	$\stackrel{\text{\ref{rr}}}{\approx}w_{t_1}...w_{t_q}...w_{t_a} w_{t_{a+1}}^{-1}...w_{t_{a+l}}^{-1}...w_{t_{a+b}}^{-1}$, i.e. we can cancel $v_\rho$ in $\hat{w}$ using \ref{xx}, \ref{vv}, and \ref{rr}, a contradiction.	
\end{proof} 
\begin{lem} \label{7}
	Let $\rho\in A$ and let $l\in\{1,...,b\}$ such that $t_{a+l}\neq m$. If $\rho\in\{(t_{a+l})_u+1,...,(t_{a+l}+1)_u-1\}$ then $\rho\in\{(t_{a+l})_u+2,...,(t_{a+l}+1)_u-1\}\subseteq A_w$.
\end{lem}
\begin{proof} It is a consequence of Lemma \ref{x middle} that  $\rho\in\{(t_{a+l})_u+2,...,(t_{a+l}+1)_u-1\}$ and by \ref{6w}, we have $\{(t_{a+l})_u+2,...,(t_{a+l}+1)_u-1\}\subseteq A_w$.
\end{proof}

\begin{lem} \label{8}
	Let $\rho\in A$. If  $t_{a+1}=m$ and $\rho\in\{m_u+1,...,n\}$ then $\rho\in\{m_u+2,...,n\}\subseteq A_w$.
\end{lem}
\begin{proof}
	Suppose $\rho=m_u+1=(t_{a+1})_u+1$.  By Lemma \ref{x middle}, we have $\rho\notin A$. Therefore, $\rho\in\{m_u+2,...,n\}\subseteq A_w$ by \ref{5W}.
\end{proof}

\begin{lem} \label{111}
	If $1<1_x<1_u$  then $\rho\notin A$ for all  $\rho\in\{1,...,1_u-1_x\}$.
\end{lem}
\begin{proof} Let $\rho\in\{1,...,1_u-1_x\}$. Assume $\rho\in A$.
	We	observe that $1_u-1_x= 2\lvert w^{-1}_{t_{a+b}}...w^{-1}_{t_{a+1}}\rvert-2\lvert w_{t_1}...w_{t_a}\rvert=2k$ for some positive integer $k$. We put $\mathcal{U}=w_{t_1}...w_{t_a}$ and $\mathcal{X}=w^{-1}_{t_{a+b}}...w^{-1}_{t_{a+1}}$, i.e. $2k=2\lvert\mathcal{X}\rvert-2\lvert\mathcal{U}\rvert$ and $\lvert\mathcal{X}\rvert=\lvert\mathcal{U}\rvert+k$. Let ${w_{t_{a+1}}^{-1}}...{w_{t_{a+b}}^{-1}}={y}_1...{y}_{\lvert\mathcal{U}\rvert}{y}_{\lvert\mathcal{U}\rvert+1}...{y}_{\lvert\mathcal{U}\rvert+k}$, where $y_1,...,y_{\lvert\mathcal{U}\rvert+k}\in\{x_1,...,x_{n-2}\}$. Then 
	$	
	v_\rho{w}_{t_1}...{w}_{t_a}{y}_1...{y}_{\lvert\mathcal{U}\rvert}{y}_{\lvert\mathcal{U}\rvert+1}...{y}_{\lvert\mathcal{U}\rvert+k}
	\stackrel{\text{\ref{vv}}}{\approx} {w}_{t_1}...{w}_{t_a}v_{\rho+2\lvert w_{t_1}...w_{t_a}\rvert} \\ {y}_1...{y}_{\lvert\mathcal{U}\rvert}{y}_{\lvert\mathcal{U}\rvert+1}...{y}_{\lvert\mathcal{U}\rvert+k}$. 
	Using Remark \ref{two}, it is routine to calculate that $2\lvert w^{-1}_{t_{a+b}}...w^{-1}_{t_{a+1}}\rvert< i_{t_{a+1}}+2\lvert w^{-1}_{t_{a+1}}\rvert$, i.e.
	$(1_u-1_x)+2\lvert w_{t_1}...w_{t_a}\rvert=   2\lvert w^{-1}_{t_{a+b}}...w^{-1}_{t_{a+1}}\rvert< i_{t_{a+1}}+2\lvert w^{-1}_{t_{a+1}}\rvert$. This implies $\rho+2\lvert w_{t_1}...w_{t_a}\rvert\leq i_{t_{a+1}}+2\lvert w^{-1}_{t_{a+1}}\rvert$. 
	Then 
	${w}_{t_1}...{w}_{t_a}v_{\rho+2\lvert w_{t_1}...w_{t_a}\rvert}{y}_1...{y}_{\lvert\mathcal{U}\rvert}{y}_{\lvert\mathcal{U}\rvert+1}...{y}_{\lvert\mathcal{U}\rvert+k}\stackrel{\text{\ref{vv}}}{\approx}
	{w}_{t_1}...{w}_{t_a}{y}_1\\...{y}_{\lvert\mathcal{U}\rvert}v_{\rho}{y}_{\lvert\mathcal{U}\rvert+1}...{y}_{\lvert\mathcal{U}\rvert+k}$. 
	Note that $1_u-1_x$ is even and there is $i\in\{2,4,...,1_u-1_x\}$   such that 
	$\rho\in\{i-1,i\}$. 
	If $\rho=i-1$ then $\rho-2\lvert{y}_{\lvert\mathcal{U}\rvert+1}...{y}_{\lvert\mathcal{U}\rvert+\frac{i}{2}-1}\rvert=1$.
	If $\rho=i$ then $\rho-2\lvert{y}_{\lvert\mathcal{U}\rvert+1}...{y}_{\lvert\mathcal{U}\rvert+\frac{i}{2}-1}\rvert=2$. Thus, ${w}_{t_1}...{w}_{t_a}{y}_1...{y}_{\lvert\mathcal{U}\rvert}v_{\rho}{y}_{\lvert\mathcal{U}\rvert+1}...{y}_{\lvert\mathcal{U}\rvert+k} \\
	\stackrel{\text{\ref{vv}}}{\approx} {w}_{t_1}...{w}_{t_a}{y}_1...{y}_{\lvert\mathcal{U}\rvert}{y}_{\lvert\mathcal{U}\rvert+1}...  v_{\rho-2\lvert{y}_{\lvert\mathcal{U}\rvert+1}...{y}_{\lvert\mathcal{U}\rvert+\frac{i}{2}-1}\rvert}{y}_{\lvert\mathcal{U}\rvert+\frac{i}{2}}...{y}_{\lvert\mathcal{U}\rvert+\frac{1_u-1_x}{2}}\\
	={w}_{t_1}...{w}_{t_a}{y}_1...{y}_{\lvert\mathcal{U}\rvert}{y}_{\lvert\mathcal{U}\rvert+1}...v_{\hat{\rho}}{y}_{\lvert\mathcal{U}\rvert+\frac{i}{2}}...{y}_{\lvert\mathcal{U}\rvert+\frac{1_u-1_x}{2}}$ (where $\hat{\rho}\in\{1,2\}$) \\
	$\stackrel{\text{\ref{rr}}}{\approx}
	{w}_{t_1}...{w}_{t_a}{y}_1...{y}_{\lvert\mathcal{U}\rvert}{y}_{\lvert\mathcal{U}\rvert+1}...{y}_{\lvert\mathcal{U}\rvert+\frac{i}{2}}...{y}_{\lvert\mathcal{U}\rvert+\frac{1_u-1_x}{2}}$, \\
	i.e. we can cancel $v_\rho$ in $\hat{w}$ using \ref{vv} and \ref{rr}, a contradiction.
\end{proof}  
\begin{lem} \label{10}
	Let $\rho\in A$ and $\rho\in\{1,...,1_u-1\}$.	If $1<1_u\leq 1_x$  then $\rho\in \{1,...,1_u-1\}\subseteq A_w$ and if $1<1_x<1_u$ then $\rho\in\{1_u-1_x+1,...,1_u-1\}\subseteq A_w$.
\end{lem}
\begin{proof} If $1<1_u\leq 1_x$ then  $\{1,...,1_u-1\}\subseteq A_w$ by \ref{7W}. 
	If $1<1_x<1_u$, it is a consequence of Lemma \ref{111} that  $\rho\in\{1_u-1_x+1,...,1_u-1\}$ and by \ref{7W}, we have $\{1_u-1_x+1,...,1_u-1\}\subseteq A_w$.
\end{proof}

\begin{lem} \label{11}
	We have $(t_q)_u\notin A$ for all $q\in\{1,...,a\}$.
\end{lem}
\begin{proof}  Let $q\in\{1,...,a\}$. We have $w_{t_q}=u_{i_{t_q}}u_{i_{t_q}+2}...u_{i_{t_q}+2\lvert w_{t_q} \rvert-2}$ and $(t_q)_u=i_{t_q}$. Assume $(t_q)_u\in A$. If $i_{t_q}\geq 2$ then  $v_{i_{t_q}}w_{t_1}...w_{t_q}...w_{t_a}w_{t_{a+1}}^{-1}...w_{t_{a+b}}^{-1}  \stackrel{\text{\ref{xx}}}{\approx}w_{t_1}...v_{i_{t_q}}u_{i_{t_q}}u_{i_{t_q}+2}...u_{i_{t_q}+2\lvert w_{t_q} \rvert-2}w_{t_{q+1}}...\\w_{t_a}w_{t_{a+1}}^{-1}...w_{t_{a+b}}^{-1}
	\stackrel{\text{\ref{zz}}}{\approx}w_{t_1}...v_{i_{t_q}+2\lvert w_{t_q} \rvert+1}u_{i_{t_q}-1}u_{i_{t_q}+1}...u_{i_{t_q}+2\lvert w_{t_q} \rvert-3}w_{t_{q+1}}...w_{t_a}w_{t_{a+1}}^{-1}...w_{t_{a+b}}^{-1}$. \\
	If $i_{t_q}=1$ then  $q=1$ and 
	$v_{i_{t_1}}w_{t_1}w_{t_2}...w_{t_a}w_{t_{a+1}}^{-1}...w_{t_{a+b}}^{-1}
	=v_1u_{1}u_{3}...u_{{1}+2\lvert w_{t_1} \rvert-2}w_{t_2}...w_{t_a}w_{t_{a+1}}^{-1}...w_{t_{a+b}}^{-1} \\
	\stackrel{\text{\ref{w}}}{\approx} v_1v_2...v_{1+2\lvert w_{t_1} \rvert+1}w_{t_{2}}...w_{t_a}w_{t_{a+1}}^{-1}...w_{t_{a+b}}^{-1}$. 
	We observe that we can replace several variables in $\hat{w}$ by variables with decreasing index by \ref{zz} and the variables $u_1,u_{3},...,u_{{1}+2\lvert w_{t_1} \rvert-2}$ were canceled  in $\hat{w}$ by \ref{w}, respectively, a contradiction.
\end{proof}

\begin{lem} \label{12}
	We have $(t_{a+l})_u\notin A$	 for all $l\in\{1,...,b\}$.
\end{lem}
\begin{proof} Let $l\in\{1,...,b\}$. Now assume that $(t_{a+l})_u\in A$. We will have the following two cases. In the first case, we suppose that there exists $q\in\{1,...,a\}$ with $t_q>t_{a+l}$ and, of course, for the trivial second case is supposed $t_q<t_{a+l}$ for all $q\in\{1,...,a\}$. Using  \ref{xx} and \ref{vv} in the first case and \ref{vv} in the second case, together with a few tedious calculations, both cases imply $v_{(t_{a+l})_u}w_{t_1}...w_{t_q}...w_{t_a}w_{t_{a+1}}^{-1}...w_{t_{a+b}}^{-1}\approx w_{t_1}...w_{t_a}v_{i_{t_{a+l}}+2\lvert w_{t_{a+1}}^{-1}...w_{t_{a+l}}^{-1}\rvert}w_{t_{a+1}}^{-1}...w_{t_{a+b}}^{-1}$. It is routine to calculate that \\ $w_{t_1}...w_{t_a}v_{i_{t_{a+l}}+2\lvert w_{t_{a+1}}^{-1}...w_{t_{a+l}}^{-1}\rvert}w_{t_{a+1}}^{-1}...w_{t_{a+b}}^{-1}\stackrel{\text{\ref{vv}}}\approx w_{t_1}...w_{t_a} w_{t_{a+1}}^{-1}...v_{i_{t_{a+l}}+2\lvert w_{t_{a+l}}^{-1}\rvert}w_{t_{a+l}}^{-1}...w_{t_{a+b}}^{-1}$. \\
	If ${i_{t_{a+l}}+2\lvert w_{t_{a+l}}^{-1}\rvert}>3$ then $w_{t_1}...w_{t_a} w_{t_{a+1}}^{-1}...v_{i_{t_{a+l}}+2\lvert w_{t_{a+l}}^{-1}\rvert}w_{t_{a+l}}^{-1}...w_{t_{a+b}}^{-1} \\
	=w_{t_1}...w_{t_a} w_{t_{a+1}}^{-1}...v_{i_{t_{a+l}}+2\lvert w_{t_{a+l}}^{-1}\rvert}x_{i_{t_{a+l}}+2\lvert w_{t_{a+l}} \rvert-2}x_{i_{t_{a+l}}+2\lvert w_{t_{a+l}} \rvert-4}...x_{i_{t_{a+l}}}w_{t_{a+l+1}}^{-1}...w_{t_{a+b}}^{-1} \\ \stackrel{\text{\ref{yy}}}\approx w_{t_1}...w_{t_a} w_{t_{a+1}}^{-1}...v_{i_{t_{a+l}}+2\lvert w_{t_{a+l}}^{-1}\rvert+1}x_{i_{t_{a+l}}+2\lvert w_{t_{a+l}} \rvert-3}x_{i_{t_{a+l}}+2\lvert w_{t_{a+l}} \rvert-5}...x_{i_{t_{a+l}}-1}w_{t_{a+l+1}}^{-1}...w_{t_{a+b}}^{-1}$. \\
	If ${i_{t_{a+l}}+2\lvert w_{t_{a+l}}^{-1}\rvert}=3$ then $w_{t_{a+b}}^{-1}=x_1$. Thus $w_{t_1}...w_{t_a}v_{i_{t_{a+l}}+2\lvert w_{t_{a+1}}^{-1}...w_{t_{a+l}}^{-1}\rvert}w_{t_{a+1}}^{-1}...w_{t_{a+b}}^{-1} \\ \stackrel{\text{\ref{vv}}}\approx w_{t_1}...w_{t_a} w_{t_{a+1}}^{-1}...w_{t_{a+b-1}}^{-1}v_3x_1 \stackrel{\text{\ref{x}}}\approx w_{t_1}...w_{t_a} w_{t_{a+1}}^{-1}...w_{t_{a+b-1}}^{-1}v_1v_2v_3v_4$. 
	We observe that we can replace several variables in $\hat{w}$ by variables with decreasing index by \ref{yy} and the variable $x_1$ can be canceled  in $\hat{w}$ by \ref{x}, respectively, a contradiction. 
\end{proof}
If we summarize the previous lemmas, then we obtain:
\begin{lem} \label{A sub W}
	We have $A\subseteq A_w$.
\end{lem}
\begin{proof}
	Let $\rho\in A$. Then it is easy to verify that $\rho\in\{1,...,1_u\}$ or $\rho\in\{k_u+1,...,(k+1)_u\}$ for some $k\in\{1,....,m-1\}$ or $\rho\in \{m_u+1,...,n\}$. Suppose that $\rho\in\{k_u+1,...,(k+1)_u-1\}$ for some $k\in\{1,...,m-1\}$.  Lemma \ref{11} and \ref{12} show that $k_u\notin A$. Then we can conclude that $\rho\in A_w$ by Lemma \ref{4} and \ref{7}. Suppose $\rho\in\{m_u+1,...,n\}$. Then we can conclude that $\rho\in A_w$ by Lemma \ref{5} and \ref{8}.	Finally, we suppose that $\rho\in\{1,...,1_u-1\}$. Then we can conclude that $\rho\in A_w$ by Lemma \ref{10}.  Eventually, we  have $\rho\in A_w$ for all $\rho\in A$. Therefore $A\subseteq A_w$.
\end{proof}
Lemma \ref{w in Q} and \ref{A sub W} prove that $\hat{w}=v_A\hat{w}_1...\hat{w}_{m}\in W_n$. Consequently, we have:				
\begin{prop}   \label{P in}
	$P \subseteq W_n$.
\end{prop}
By the definition of the set $P$  and Proposition \ref{P in}, it is proved:
\begin{cor}   \label{co 1}
	Let $w\in X_n^*$. Then there is $w'\in P\subseteq W_n$ with $w\approx w'$.
\end{cor}
\section{A Presentation for $IOF_n^{par}$}
In this section, we exhibit a  presentation for $IOF_n^{par}$. Concerning the results from the previous sections, it remains to show that $\lvert W_n\rvert \leq \lvert IOF_n^{par}\rvert$. For this, we construct the word $w_\alpha,$ for all $\alpha\in IOF_n^{par}$,  in the following way.  Let  $\alpha = \bigl(\begin{smallmatrix}
d_1 &<&d_2&<&   \cdots & < &d_p \\
m_1 & &m_2 & & \cdots &    & m_p
\end{smallmatrix}\bigr) \in IOF_n^{par}\backslash\{\varepsilon\}$ for a positive integer $p\leq n$. There are a unique $l\in\{0,1,...,p-1\}$ and a unique set $\{r_1,...,r_l\}\subseteq \{1,...,p-1\}$ such that (i)-(iii) are satisfied:\\
\indent (i) $r_1<...<r_l$;\\
\indent (ii) $d_{r_i+1}-d_{r_i}\neq m_{r_i+1}-m_{r_i}$ for $i\in\{1,...,l\}$; \\
\indent (iii) $d_{i+1}-d_{i}= m_{i+1}-m_{i}$ for $i\in\{1,...,p-1\}\backslash\{r_1,...,r_l\}$.  \\ 
Note that $l=0$ means $\{r_1,...,r_l\}=\emptyset$. Further, we put $r_{l+1} = p$.
For $i\in\{1,...,l\}$, we define
\begin{align*}
w_i 
=& \begin{cases}
x_{m_{r_i},\frac{(m_{r_i+1}-m_{r_i})-(d_{r_i+1}-d_{r_i})}{2}} \ \mbox{if} \ m_{r_i+1}-m_{r_i}>d_{r_i+1}-d_{r_i}; \\
u_{d_{r_i},\frac{(d_{r_i+1}-d_{r_i})-(m_{r_i+1}-m_{r_i})}{2}} \ \mbox{if} \ m_{r_i+1}-m_{r_i}<d_{r_i+1}-d_{r_i}. 
\end{cases}
\end{align*} Obviously, we have $w_i\in W_x\cup W_u$ for all $i\in\{1,...,l\}$. If $m_p=d_p$ then we put $w_{l+1}=\epsilon$. If $m_p \neq d_p$, we define additionally
\begin{align*}
w_{l+1} 
=& \begin{cases}
x_{m_{p},\frac{d_p-m_p}{2}} \ \mbox{if} \ d_p>m_p ; \\
u_{d_{p},\frac{m_p-d_p}{2}} \ \mbox{if} \ d_p<m_p. 
\end{cases}&&\qedhere
\end{align*} 
Clearly, $w_{l+1}\in W_x\cup W_u$.  We consider the word 
\begin{center}
	$w=w_1...w_{l+1}$.
\end{center} From this word, we construct a new word $w_\alpha^*$ by arranging the subwords $s\in W_x$ in reverse order at the end, replacing $s$ by $s^{-1}$. In other words, we consider the word \begin{center}
	$w_\alpha^*= w_{s_1}...w_{s_a}w_{s_{a+1}}^{-1}...w_{s_{a+b}}^{-1}$
\end{center} such that $w_{s_1},...,w_{s_a} \in W_u$, $w_{s_{a+1}},...,w_{s_{a+b}}\in W_x$ and $\{w_{s_1},...,w_{s_a},w_{s_{a+1}},...,w_{s_{a+b}}\}=\{w_1,...,w_{a+b}\}$, where $s_1<...<s_a, s_{a+b}<...<s_{a+1}$, and $a,b \in \overline{n}\cup\{0\}$ with
\begin{center}
	$a+b=
	\begin{cases}
	l  & \mbox{if} \ d_p=m_p;   \\
	l+1 & \mbox{if} \ d_p\neq m_p.
	\end{cases}$
\end{center} For convenient, $a=0$ means $w_\alpha^*=w_{s_{a+1}}^{-1}...w_{s_{a+b}}^{-1}  $ and $b=0$ means $w_\alpha^*=w_{s_1}...w_{s_a}$.  Now, we add recursively letters from the set $\{v_1,...,v_n\}\subseteq X_n$ to the word $w_\alpha^*$, obtaining new words $\lambda_0,\lambda_1,...,\lambda_p$. \\
\noindent (1)	For $d_p \leq n-2$: \\
\indent	(1.1) if $m_p<d_p$ then $\lambda_0= v_{d_p+2}...v_nw^*_\alpha$; \\
\indent	(1.2) if $n-1>m_p>d_p$ then $\lambda_0= v_{m_p+2}...v_nw^*_\alpha$; \\ 
\indent	(1.3) if $m_p=d_p$ then $\lambda_0= v_{m_p+1}...v_nw^*_\alpha$; \\
\indent otherwise $\lambda_0=w^*_\alpha$. \\
(2)	 If  $d_p=m_p=n-1$ then $\lambda_0=v_nw^*_\alpha$.  Otherwise $\lambda_0=w^*_\alpha$.  \\
(3)	 For $k\in\{2,...,p\}$: \\
\indent(3.1) if $2\leq m_k-m_{k-1}=d_k-d_{k-1}$ then $\lambda_{p-k+1}= v_{d_{k-1}+1}...v_{d_k-1}\lambda_{p-k}$; \\
\indent(3.2) if $2< m_k-m_{k-1}<d_k-d_{k-1}$ then $\lambda_{p-k+1} = v_{d_k-(m_k-m_{k-1}-2)}...v_{d_k-1}\lambda_{p-k}$; \\
\indent	(3.3)	if $m_k-m_{k-1}>d_k-d_{k-1}>2$ then $\lambda_{p-k+1}=v_{d_{k-1}+2}...v_{d_k-1}\lambda_p$; \\
\indent otherwise $\lambda_{p-k+1}=\lambda_{p-k}$. \\
(4)	If $d_1=1$ or $m_1=1$ then $\lambda_p=\lambda_{p-1}$. \\
(5)	 If $1<d_1\leq m_1$ then $\lambda_p=v_1...v_{d_1-1}\lambda_{p-1}$. \\
(6)	 If $1<m_1<d_1$ then $\lambda_p=v_{d_1-m_1+1}...v_{d_1-1}\lambda_{p-1}$. \\

\indent The word $\lambda_p$ induces a set  $A=\{a\in\overline{n}: v_a \ \mbox{is a variable in}\ \lambda_p\}$ and it is easy to verify that  $\rho\notin A$ for all $\rho\in dom(\alpha)$. We put $w_\alpha=\lambda_p$. The word $w_\alpha$ has the form $w_\alpha=v_Aw_\alpha^*$. \\

Our next aim is to present the relationship between the cardinality of $W_n$ and $IOF_n^{par}$. This leading us to assume the existence of a map  $f: IOF_n^{par}\backslash \{\varepsilon\}\rightarrow W_n \backslash \{v_{\overline{n}}\}$,
where $f(\alpha)=w_\alpha$ for all $\alpha\in IOF_n^{par}\backslash \{\varepsilon\}$. We start at the top by constructing the transformation $\alpha_{v_Aw^*}$ for any  $v_Aw^*\in W_n$ different from $v_{\overline{n}}$. Let $v_Aw^*\in W_n\backslash \{v_{\overline{n}}\}$. We have  $w\in Q_0, A\subseteq A_w$, and there are $w_1,...,w_m\in W_u\cup W_x$ such that $w=w_1...w_m$ for some positive integer $m$. For $k\in\{1,...,m\}$, we define $a_k=k_u+2$  and $b_k=i_k+2j_k+2$, whenever $w_k\in W_x$. On the other hand, we define $a_k=i_k+2j_k+2$  and $b_k= k_x+2$, whenever $w_k\in W_u$. It is easy to verify that $a_m=b_m$. We put
\begin{center}
	$\alpha_{v_Aw^*}=\overline{v}_A 
	\begin{pmatrix}
	1+1_u-min\{1_u,1_x\}...1_u& a_1...2_u & \cdots & a_{m-1}...m_u & a_m...n \\
	1+1_x-min\{1_u,1_x\}...1_x& b_1...2_x & \cdots & b_{m-1}...m_x & b_m...n
	\end{pmatrix}$. 
\end{center}
For convenience, we also give 
$\alpha_{v_Aw^*}= \bigl(\begin{smallmatrix}
d_1 & &d_2& &   \cdots &   &d_p \\
m_1 & &m_2 & & \cdots &    & m_p
\end{smallmatrix}\bigr)$ for some positive integer $p\leq n$.
In the following, we show that $\alpha_{v_Aw^*}$ is well-defined in the sense that the construction of $\alpha_{v_Aw^*}$ gives a transformation.
\begin{lem} \label{16}
	$\alpha_{v_Aw^*}$ is well-defined.
\end{lem}
\begin{proof}
	Let $k\in\{1,...,m-1\}$.  
	Suppose $w_k, w_{k+1}\in W_u$. We have $k_u=i_k$, $k_x=i_k+2\lvert w_k \rvert+2\lvert W_u^k\rvert-2\lvert W_x^k\rvert, (k+1)_u=i_{k+1}, (k+1)_x=i_{k+1}+2\lvert w_{k+1}\rvert+2\lvert W_u^{k+1}\rvert-2\lvert W_x^{k+1}\rvert$, and $a_k=i_k+2j_k+2, b_k=k_x+2$.   
	Then $(k+1)_u-a_k=i_{k+1}-(i_k+2j_k+2)$ and $(k+1)_x-b_k=i_{k+1}+2\lvert w_{k+1}\rvert+2\lvert W_u^{k+1}\rvert-2\lvert W_x^{k+1}\rvert-k_x-2=i_{k+1}+2\lvert w_{k+1}\rvert+2\lvert W_u^{k+1}\rvert-2\lvert W_x^{k+1}\rvert-i_k-2\lvert w_k \rvert-2\lvert W_u^k\rvert+2\lvert W_x^k\rvert-2=i_{k+1}-i_k-2j_k-2=i_{k+1}-(i_k+2j_k+2)$. Therefore, $(k+1)_u-a_k=(k+1)_x-b_k$. For the rest cases ($w_k\in W_u$ and $w_{k+1}\in W_x$, $w_k\in W_x$ and $w_{k+1}\in W_u$ as well as $w_k, w_{k+1}\in W_x$), a proof similar as above will eventually show that $(k+1)_u-a_k=(k+1)_x-b_k$. Furthermore,
	suppose $d_p=m_p$. Let $k\in\{1,...,m\}$ and $w_k\in W_u$.  We have $a_k-k_u=i_k+2j_k+2-k_u=i_k+2j_k+2-i_k=2j_k+2$ 
	and $b_k-k_x=k_x+2-k_x=2$. Thus, $a_k-k_u\neq b_k-k_x$. For the case $w_k\in W_x$, we can show $a_k-k_u\neq b_k-k_x$ in the same way.   Continuously, suppose $d_p\neq m_p$. By the previous part of the proof, we have $a_k-k_u\neq b_k-k_x$ for all $k\in\{1,...,m-1\}$. Moreover, we observe that $d_p\notin\{a_m,...,n\}$ and $m_p\notin\{b_m,...,n\}$ because $n-a_m=n-b_m$. This implies, $d_p=m_u$ and $m_p=m_x$. By any of the above, we can conclude that $\alpha_{v_Aw^*}$ is well-defined. 
\end{proof}

The proof of Lemma \ref{16} shows $(k+1)_u-a_k=(k+1)_x-b_k$ for all $k\in\{1,...,m-1\}$.  $a_k-k_u\neq b_k-k_x$ for all $k\in\{1,...,m\}$, whenever $d_p=m_p$ and  $a_k-k_u\neq b_k-k_x$  for all $k\in\{1,...,m-1\}$ and $d_p=m_u, m_p=m_x$, whenever $d_p\neq m_p$. Furthermore, observing by trivial calculation, $a_k-k_u\geq 2$ and  $b_k-k_x\geq 2$. Therefore, if there exists $i\in\{1,...,p-1\}$, where $d_{i+1}-d_i\neq m_{i+1}-m_i$, then $d_i\in\{1_u,...,(m-1)_u\}(\cup\{m_u\})$, $m_i\in\{1_x,...,(m-1)_x\}(\cup\{m_x\})$ and we put $k_u=d_{r_k}, k_x=m_{r_k}$ for all $k\in\{1,...,m-1\}(\cup\{m\})$ (we put $r_m=p$, whenever $d_p\neq m_p$).   This gives the unique set $\{r_1,...,r_m\}$ as required the definition of $w_{\alpha_{v_Aw^*}}$. Moreover, it need to show that $\alpha_{v_Aw^*}\in IOF_n^{par}\backslash \{\varepsilon\}$ by checking  (i)-(iv) of  Proposition \ref{4 choice}. We will now show that $\alpha_{v_Aw^*}\in IOF_n^{par} $ as well as $w_{\alpha_{v_Aw^*}}=v_Aw^*$. This gives the tools to calculate that $\lvert W_n \rvert \leq \lvert IOF_n^{par} \rvert$.
\begin{lem} \label{17}
	$\alpha_{v_Aw^*}\in IOF_n^{par}\backslash \{\varepsilon\}$.
\end{lem}
\begin{proof}
Clearly $\alpha_{v_Aw^*}\neq \varepsilon$.	We will prove that $\alpha_{v_Aw^*}$ is satisfied (i)-(iv) in Proposition \ref{4 choice}.	We observe that $d_1<d_2<\cdot\cdot\cdot<d_p$ and $m_1<m_2<\cdot\cdot\cdot<m_p$  by definition of $\alpha_{v_Aw^*}$. We have $1_u-d_1=1_x-m_1$, i.e.  $1_u-1_x=d_1-m_1$. By the definition of $k_u$ and $k_x$, for  $k\in\{1,...,m\}$, we observe that $1_u-1_x$ is even, i.e. $d_1-m_1$ is even. Thus, $d_1$ and $m_1$ have the same parity.   Let $d_{i+1}-d_i=1$ for some $i\in\{1,...,p-1\}$. Then $d_i\in dom(\alpha)\backslash\{1_u,...,m_u\}$ implies $m_{i+1}-m_i=d_{i+1}-d_i=1$. 
	Let $m_{i+1}-m_i=1$ for some $i\in\{1,...,p-1\}$. Then $m_i\in im(\alpha)\backslash\{1_x,...,m_x\}$ implies $d_{i+1}-d_i=m_{i+1}-m_i=1$.  Let $d_{i+1}-d_i$ is even. Suppose $d_{i+1}-d_i\neq m_{i+1}-m_i$. This gives $d_i=k_u$ and $m_i=k_x$ for some $k\in\{1,...,m-1\}$. By the definition of $k_u$ and $k_x$, we observe that $k_u-k_x$ is even. Moreover, $(k+1)_u-d_{i+1}=(k+1)_x-m_{i+1}$ since $(k+1)_u- (k+1)_x$ is even, we have $d_{i+1}-m_{i+1}$ is even. Then $d_{i+1}$, $d_i$ and $d_i$, $m_i$ as well as $d_{i+1}, m_{i+1}$ have the  same parity. This implies $m_{i+1}, m_i$ have the same parity, i.e. $m_{i+1}-m_i$ is even. Conversely, we can prove  similarly that, if $m_{i+1}-m_i$ is even then $d_{i+1}-d_i$ is even. By Proposition \ref{4 choice}, we get $\alpha_{v_Aw^*}\in IOF_n^{par}$. 
\end{proof}
Lemma \ref{17} shows $\alpha_{v_Aw^*}\in IOF_n^{par}\backslash \{\varepsilon\}$. We can construct $f(\alpha_{v_Aw^*})=w_{\alpha_{v_Aw^*}}$ where $w_{\alpha_{v_Aw^*}}=v_{\tilde{A}}\hat{w}^*_{\alpha_{v_Aw^*}}$ with $\hat{w}=\hat{w}_1...\hat{w}_{m}$ for $\hat{w}_1,...,\hat{w}_{m}\in W_u\cup W_x$ and $\tilde{A}\subseteq \overline{n}$. We will prove $f$ is surjective in the next Lemma.

\begin{lem} \label{21}
	Let $v_Aw^*\in W_n\backslash \{v_{\overline{n}}\}$. Then there is $\alpha\in IOF_n^{par}\backslash \{\varepsilon\}$ with $v_Aw^*=w_{\alpha}$.
\end{lem}
\begin{proof}
	We have $w_{\alpha_{v_Aw^*}}=v_{\tilde{A}}\hat{w}^*_{\alpha_{v_Aw^*}}$ where
	$\hat{w}=\hat{w}_1...\hat{w}_{m}$ with $\hat{w}_1,...,\hat{w}_{m}\in W_u\cup W_x$ and $\tilde{A}\subseteq \overline{n}$. First, our goal is to show that $\hat{w}=w$.  Suppose $d_p=m_p$ and let $k\in\{1,...,m\}$ such that $b_k-k_x>a_k-k_u$. By definition of $\hat{w}_k$, we have $\hat{w}_k=x_{k_x,\frac{(b_k-k_x)-(a_k-k_u)}{2}}$ and $k_x=i_k$. Then $\frac{(b_k-k_x)-(a_k-k_u)}{2}=\frac{i_k+2j_k+2-i_k-k_u-2+k_u}{2}=j_k$, i.e. $\hat{w}_k=x_{i_k,j_k}=w_k$. For the case $b_k-k_x<a_k-k_u$, we can prove that $\hat{w}_k=w_k$ in a similar way. This gives $\hat{w}_1...\hat{w}_{m}=w_1...w_m$. \\
	Suppose $d_p\neq m_p$.  We have $a_k-k_u\neq b_k-k_x$ for all $k\in\{1,...,m-1\}$ and by a similar proof as above, we have $\hat{w}_1...\hat{w}_{m-1}=w_1...w_{m-1}$. If $m_p<d_p$ then $\hat{w}_{m}=x_{m_p,\frac{d_p-m_p}{2}}$ and $m_p=m_x=i_m$. Then $\frac{d_p-m_p}{2}=\frac{m_u-m_x}{2}=\frac{i_m+2j_m-i_m}{2}=j_m$, i.e. $\hat{w}_{m}=x_{i_m,j_m}=w_m$. For the case   $m_p>d_p$, we can prove  that $\hat{w}_{m}=w_m$ in similar way. Thus, $\hat{w}_1...\hat{w}_{m-1}\hat{w}_{m}=w_1...w_{m-1}w_m$. 	Then $ w=\hat{w}$, i.e. $w^*=\hat{w}^*_{\alpha_{v_Aw^*}}$. The next goal is to show that $A=\tilde{A}$. \\
	
	\noindent	1) To show that $A\subseteq\tilde{A}$:	Let $a\in A$. We have $A\subseteq A_w$ since $v_Aw^*\in W_n$. Therefore, we have the following cases: 
	$a\in\{a_m,...,n\}=A_1$ or $a\in\{a_k,...,(k+1)_u-1\}=A_2$ for some $k\in\{1,...,m-1\}$ or $a\in\{1+1_u-min\{1_u,1_x\},...,1_u-1\}=A_3$. 
	If $a\in A_1$ and $m_p\neq d_p$ then $a\in \tilde{A}$ since (1.1) and (1.2), respectively. If $a\in A_1$  and $a\in\{d_p+1,...,n\}$ with $m_p=d_p$ then $a\in \tilde{A}$ since (1.3) and (2), respectively. 
	Suppose $a\in A_2$ with $a\in\{a_k,...,d_{r_k+1}-1\}$. If $2< d_{r_k+1}-d_{r_k} < m_{r_k+1}-m_{r_k}$ then $w_k\in W_x$. Note $a_k=k_u+2=d_{r_k}+2$. Thus, $a\in \tilde{A}$ since (3.3). If $2< m_{r_k+1}-m_{r_k} < d_{r_k+1}-d_{r_k}$ then $w_k\in W_u$. Note $d_{r_k+1}-a_k=m_{r_k+1}-b_k, b_k=k_x+2$, and  $a_k=a_k-b_k+b_k=d_{r_k+1}-m_{r_k+1}+k_x+2=d_{r_k+1}-m_{r_k+1}+m_{r_k}+2$. Thus, $a\in \tilde{A}$ since (3.2). 
	Suppose $a\in A_3$. If $1<d_1\leq m_1$ and $a\in\{1,...,d_1-1\}$ then $a\in \tilde{A}$ since (5).	If $1<m_1<d_1$ and $a\in\{d_1-m_1+1,...,1_u-1\}$ then $a\in \tilde{A}$ since (6) (note that $1_u-1_x=d_1-m_1$).  Suppose $a\in A_1\cup A_2\cup A_3$ and  there exists $s\in\{2,...,p\}$ such that $d_s-d_{s-1}=m_s-m_{s-1}\geq 2$  with
	$a\in\{d_{s-1}+1,...,d_s-1\}$ then  $a\in \tilde{A}$ since (3.1). By any of the above, we have $A\subseteq\tilde{A}$. \\
	
	\noindent	2) To show that $\tilde{A}\subseteq A$: Let $A_1=\{1+1_u-min\{1_u,1_x\},...,1_u-1\} , A_2=\{a_1,...,2_u-1\}\cup\{a_2,...,3_u-1\}\cup...\cup\{a_{m-1},...,m_u-1\},\mbox{and} \ A_3=\{a_m,...,n\}$. Because $A\subseteq A_w$, we have $A\subseteq A_1\cup A_2\cup A_3$ and $A\cap\{d_1,...,d_p\}=\emptyset$. This implies $A\subseteq A_1\cup A_2\cup A_3\backslash \{d_1,...,d_p\}$. Conversely, we have $A_1\cup A_2\cup A_3\backslash \{d_1,...,d_p\}\subseteq A$ by the definition of $\alpha_{v_Aw^*}$. Thus, $A=A_1\cup A_2\cup A_3\backslash \{d_1,...,d_p\}$. \\
	Let $a\in\tilde{A}$. By the definition of $\tilde{A}$, we can observe that $a\neq d_i$ for all $i\in\{1,...,p\}$.
	Suppose $a$ is given by (1.1) or (1.2) or (1.3) or (2). Then $a\in A_3\backslash\{d_1,...,d_p\}$.
	Suppose $a$ is given by (3.1). Then $a\in A_1\cup A_2\cup A_3\backslash \{d_1,...,d_p\}$. 
	Suppose $a$ is given by (3.2), i.e. $a\in\{d_s-m_s+m_{s-1}+2,...,d_s-1\}$ for some $s\in\{2,...,p\}$.  We have already shown that there is $k\in\{1,...,m-1\}$ such that $d_s-m_s+m_{s-1}+2=a_k$. Then $a\in A_2\backslash\{d_1,...,d_p\}$. 
	Suppose $a$ is given by (3.3). Then $a\in A_2\backslash\{d_1,...,d_p\}$. 
	Suppose $a$ is given by (5). Then $a\in A_1\backslash\{d_1,...,d_p\}$. 
	Suppose $a$ is given by (6). Then $a\in A_1\backslash\{d_1,...,d_p\}$ (note $d_1-m_1=1_u-1_x$). 
	Therefore, we have $a\in A$, i.e. $\tilde{A}\subseteq  A$. \\
	By 1) and 2), we get $A=\tilde{A}$. This implies $v_Aw^*=v_{\tilde{A}}\hat{w}^*=w_{\alpha_{v_Aw^*}}$.      
\end{proof}
Lemma \ref{21} provides that $f$ is surjective. This gives $\lvert W_n \rvert \leq \lvert IOF_n^{par}\rvert$. We adapt now our alphabet and relations to Theorm \ref{kc} and observe by following: As already mentioned,    $\overline{X}_n=\{\overline{s}:s\in X_n\}$ is a generating set for the monoid $IOF_n^{par} $. From Lemma \ref{1}, we can conclude $\overline{X}_n$ satisfies all  the relations from $\overline{R}=\{\overline{s}_1\approx \overline{s}_2: s_1\approx s_2\in R\}$. Corollary  \ref{co 1} shows that for all $w\in\overline{X}_n^*$, there is $w'\in \overline{W}_n $ such that $w\approx w'$ is a consequence of $\overline{R}$   and consequently, $\overline{R}\subseteq \overline{X}_n^* \times \overline{X}_n^*$ and  $\overline{W}_n \subseteq \overline{X}_n^*$ satisfy the conditions 1.-3.  in Theorem  \ref{kc}. We now have all that we need to complete the  main result.
\begin{thm}
	$\langle \overline{X}_n \ \vert \ \overline{R} \rangle$ is a monoid presentation for $IOF_n^{par} $.
\end{thm}

\bibliography{sn-bibliography}

\end{document}